\documentclass[12pt]{article}

\usepackage{amsmath,amssymb,latexsym, amsfonts, amscd, amsthm, mathrsfs, xy,comment}
\xyoption{all}

\usepackage{color}

\theoremstyle{plain}
\newtheorem{mconjecture}{Conjecture}
\newtheorem{ithm}{Theorem}
\newtheorem{iprop}{Proposition}
\newtheorem{theorem}{Theorem}[section]

\newtheorem{prop}[theorem]{Proposition}

\newtheorem{lemma}[theorem]{Lemma}

\theoremstyle{definition}

\newtheorem{remark}[theorem]{Remark}

\long\def\symbolfootnote[#1]#2{\begingroup
\def\thefootnote{\fnsymbol{footnote}}\footnote[#1]{#2}\endgroup}

\def\alg{{\mathrm{alg}}}
\def\an{{\mathrm{an}}}
\def\lra{\longrightarrow}

\def\GL{{\bf GL}}

\def\A{\mathbf{A}}

\def\sR{\mathscr{R}}

\def\N{\mathrm{N}}
\def\1{\mf{1}}

\DeclareMathOperator{\cts}{cts}

\DeclareMathOperator{\chr}{char}
\DeclareMathOperator{\Frac}{Frac}
\DeclareMathOperator{\res}{res}

\DeclareMathOperator{\cyc}{cyc}

\DeclareMathOperator{\cond}{cond}

 \DeclareMathOperator{\Norm}{Norm}
\DeclareMathOperator{\rec}{rec} 
\DeclareMathOperator{\Hom}{Hom} \DeclareMathOperator{\End}{End}

\DeclareMathOperator{\ord}{ord}

 \DeclareMathOperator{\Cl}{Cl}

 \DeclareMathOperator{\Frob}{Frob}

\DeclareMathOperator{\tr}{tr}

\newcommand{\mat}[4]{\left( \begin{array}{cc} {#1} & {#2} \\ {#3} & {#4}
\end{array} \right)}

\newcommand{\mf}{\mathfrak }

\newcommand{\mscr}{\mathscr }

\def\fa{\mathfrak{a}}
\def\fn{\mathfrak{n}}
\def\fp{\mathfrak{p}}
\def\fq{\mathfrak{q}}

\def\fr{\mathfrak{r}}
\def\fl{\mathfrak{l}}
\def\fy{\mathfrak{y}}
\def\fP{\mathfrak{P}}

\def\fm{\mathfrak{m}}

\def\T{\mathbf{T}}
\def\I{\mathbf{I}}
\def\Z{\mathbf{Z}}

\def\Q{\mathbf{Q}}
\def\C{\mathbf{C}}
\def\G{\mathbf{G}}

\def\bdf{\begin{defn}}
\def\edf{\end{defn}}
\def\cM{\mathcal{M}}

\def\cL{\mathcal{L}}
\def\cH{\mathcal{H}}

\def\cO{\mathcal{O}}

\def\cS{\mathcal{S}}
\def\cw{\mathcal{W}}

\def\cW{\mathcal{W}}

\def\fb{\mathfrak{b}}

\def\Gal{{\rm Gal}}

\def\ab{{\rm ab}}

\def\Fitt{{\rm Fitt}}

\def\cF{{\cal F}}

\def\ab{\text{ab}}

\def\sL{{\mscr L}}

\def\sF{\mathscr{F}}
\def\sE{\mathscr{E}}
\def\sG{\mathscr{G}}
\def\sH{\mathscr{H}}

\begin{document}
\baselineskip 15.8pt

\title{On the Gross--Stark Conjecture}
\author{Samit Dasgupta\footnote{Supported by NSA Mathematical Sciences Program Grant H98230-15-1-0321}
\\ Mahesh Kakde\footnote{Supported by EPSRC First Grant EP/L021986/1}
 \\ Kevin Ventullo}

\maketitle

\begin{abstract}
In 1980, Gross conjectured a formula for the expected leading term at $s=0$  of the Deligne--Ribet $p$-adic $L$-function associated to a totally even character $\psi$ of a totally real field $F$.  The conjecture states that after scaling by $L(\psi \omega^{-1}, 0)$, this value is equal to a $p$-adic regulator of units in the abelian extension of $F$ cut out by $\psi \omega^{-1}$.  In this paper, we prove Gross's conjecture.  \end{abstract}

\tableofcontents

\section{Introduction}

In 1980, Gross stated a beautiful and precise analog of Stark's conjecture for the  behavior of $p$-adic $L$-functions at $s=0$ (\cite{gross}).    Let $F$ be a totally real field and let 
 \begin{equation} \label{e:chidef}
  \chi \colon G_F \longrightarrow \overline{\Q}^* 
  \end{equation}
  be a totally odd character of the absolute  Galois group of $F$.
Let $p$ be a prime integer.  We fix once and for all embeddings $\overline{\Q} \hookrightarrow \C$ and $\overline{\Q} \hookrightarrow \C_p$, so $\chi$ may be viewed as taking values in $\C$ or $\C_p$.  Here $\C_p$ denotes the  completion of an algebraic closure of $\Q_p$.  Consider the $L$-function associated to $\chi$ with Euler factors at primes above $p$ removed:
\[
L^*(\chi, s) =  L(\chi, s) \cdot \prod_{\fp \mid p} (1 - \chi(\fp) (\N\fp)^{-s}).
\]

Let $\omega\colon G_F \lra \mu_{p-1}$ (or $\mu_2$, if $p=2$) denote the Teichm\"uller character.
There is a unique meromorphic (and as long as $\chi \neq \omega^{-1}$, analytic) $p$-adic $L$-function 
\[ L_p(\chi\omega, s)\colon \Z_p \lra \C_p \] determined by the interpolation
property
\begin{equation} \label{e:padicL}
 L_p(\chi \omega, n) = L^*(\chi\omega^n, n) \text{ for } n \in \Z^{\le 0}. 
 \end{equation}
A classical theorem of Siegel implies that the values $L^*(\chi\omega^n, n)$ for $n \in \Z^{\le 0}$ are algebraic.  Hence by our fixed embedding $\overline{\Q} \hookrightarrow \C_p$, we can view these values as $p$-adic numbers. The existence of the $p$-adic $L$-function satisfying the interpolation property (\ref{e:padicL})
  was proved independently by Deligne--Ribet \cite{dr}, Cassou-Nogu\`es \cite{cn}, and Barsky \cite{b} in the 1970s, and new approaches have been considered recently in \cite{pcsd}, \cite{spiess} and \cite{bkl}.

We partition the set of primes above $p$ in $F$ as $R \cup R'$, where
\[ R = \{ \fp \mid p: \chi(\fp) = 1 \}, \qquad R' = \{ \fp \mid p: \chi(\fp) \neq 1 \}. \]
Note that $r_p(\chi) = \#R$ is precisely the number of Euler factors above $p$ in the expression 
\[
L_p(\chi \omega, 0) =  L^*(\chi, 0) = L(\chi, 0) \cdot \prod_{\fp \mid p} (1 - \chi(\fp))
\]
that vanish.  Since $\chi$ is totally odd, we  have $L(\chi, 0) \neq 0$.
Motivated by this, Gross stated the following conjecture regarding the order of vanishing of $L_p(\chi\omega, s)$ at $s=0$.

\begin{mconjecture}[Gross]  \label{c:order}  We have \[ \ord_{s=0} L_p(\chi\omega, s) = r_p(\chi).  \]
\end{mconjecture}

The inequality \begin{equation} \label{e:ineq}
 \ord_{s=0} L_p(\chi\omega, s) \ge r_p(\chi) \end{equation}
can be shown to follow from Wiles's proof of the Main Conjecture of Iwasawa theory, at least for $p \neq 2$ (for example, see \cite[\S2.1]{v}).  Recently
a more direct analytic proof of (\ref{e:ineq}) that holds for all $p$ was given  in \cite[Theorem 3]{pcsd} and \cite{spiess}.  Note that both of these latter papers use Spiess's  results on cohomological $p$-adic $L$-functions proved in \cite{spiesshilb}. 

\bigskip

Even more strikingly, Gross stated a $p$-adic analog of Stark's conjecture that gives an  exact formula for the leading term of $L_p(\chi\omega, s)$ at $s=0$.
To state this conjecture,  we first recall Gross's $p$-adic regulator $\sR_p(\chi)$. 

  Let $H$ denote the CM, cyclic extension of $F$ cut out by $\chi$, i.e. the subfield of $\overline{F}$ fixed by the kernel of $\chi$.  Let $c$ denote the unique complex conjugation of $H$.
  Let 
 \[ \log_p \colon \Q_p^* \longrightarrow \Z_p \]
 denote Iwasawa's $p$-adic logarithm, normalized such that $\log_p(p) = 0$.
 If $\fP$ is a prime ideal of $\cO_H$ lying above $p$, we consider  two continuous homomorphisms
 \begin{align*}
o_\fP= \ord_{\fP} \colon& \ F_{\fP}^* \longrightarrow \Z, \\
\ell_\fP= \log_p \circ  \Norm_{F_{\fP}/\Q_p} \colon& \ F_{\fP}^* \longrightarrow \Z_p.
\end{align*}
 
 Let $U = \cO_H[1/p]^*$ denote the group of $p$-units of $H$ and let $X$ be the free abelian group on the set $S_p$ of prime ideals of $\cO_H$ lying above $p$.  The abelian groups $U$ and $X$ are naturally modules for the group $G = \Gal(H/F)$.
 We consider the minus subspaces of these modules for the action of complex conjugation:
 \[ U^-  = \{u \in U: c(u) = u^{-1} \}, \qquad X^- = \{x \in X: c(x) = -x \}. \]
Consider the two $G$-module homomorphisms
 \begin{alignat*}{2}
 o_p \colon& U^- \longrightarrow X^-  \qquad && o_p(u) = (o_\fP(u))_{\fP \in S_p}, \\
 \ell_p \colon& U^- \longrightarrow X^- \otimes \Z_p  \qquad && \ell_p(u) = (-\ell_\fP(u))_{\fP \in S_p}.
 \end{alignat*}
 
 One  verifies that after tensoring with $\Q$, the map $o_p$ induces a $\Q[G]$-module isomorphism
 \begin{equation} \label{e:oisom}
   \xymatrix{ U^- \otimes \Q  \ar[r]^{\sim} & X^- \otimes \Q} 
   \end{equation}
    (see for example \cite[I.4]{tatebook}).
Denote by $E$ the finite extension of $\Q_p$ generated by the values of the character $\chi$.  We consider the $\chi^{-1}$-components of $U^-$ and $X^-$:
 \[ U_\chi  = \{ u \in U^- \otimes E: \sigma(u) = u^{\chi^{-1}}(\sigma) \}, \qquad
 X_\chi = \{x \in X^- \otimes E: \sigma(u) = \chi^{-1}(\sigma) x \}. \]
The $E$-vector space $X_\chi$ has dimension $r_p(\chi)$, and by (\ref{e:oisom}) the same is true for $U_\chi$.
After tensoring with $E$ (over $\Z$ and $\Z_p$ respectively), the maps $o_p$ and $\ell_p$ induce $E[\G]$-module homomorphisms
\[ o_p^\chi, \ell_p^\chi \colon U_\chi \longrightarrow X_\chi, \]
with $o_p^\chi$ an isomorphism.  In parallel with the classical Stark regulator (see \cite[I.4.5]{tatebook}), Gross's regulator is defined by\footnote{This definition of $\sR_p(\chi)$ differs from the regulator
$R_p(\chi)$ defined in \cite{gross} by the simple factor $(-1)^{r_p(\chi)}\prod_{\fp \mid p} f_\fp$, with notation as in {\em loc.\ cit}.  We have chosen our conventions to agree with \cite{ddp} in order to make the statement of Theorem~\ref{t:main} as clean as possible.}
\[ \sR_p(\chi) = \det(\ell_p^\chi \circ (o_p^\chi)^{-1}) \in E. \]

The following is often referred to as the Gross--Stark Conjecture. For simplicity we write $r$ for $r_p(\chi)$.
\begin{mconjecture}[Gross] \label{c:gross} We have:
\begin{equation} \label{e:gross}
 \frac{L_p^{(r)}(\chi, 0)}{r! L(\chi, 0)} =  \sR_p(\chi) \prod_{\fp \in R'} (1 - \chi(\fp)).
 \end{equation}
 \label{e:g2}
\end{mconjecture}

The equality (\ref{e:gross}) takes place in the field $E$.  The statement of Conjecture~\ref{c:gross} does not rely on Conjecture~\ref{c:order}. Federer and Gross proved that
when the order of $\chi$ divides $p-1$, the $p$-adic valuations of the two sides in Conjecture~\ref{c:gross} are equal  using  the Iwasawa Main Conjecture \cite[Proposition 3.10]{federergross}; in particular it follows that under this restrictive condition Conjecture~\ref{c:order} is equivalent to the statement $\sR_p(\chi) \neq 0$.\footnote{We thank John Coates for informing us about this paper.}  Further partial evidence has been discovered recently; see for instance \cite[Theorems 3.1 and 5.2]{burns}.

For notational simplicity,  define
\[
\sL_{\an}(\chi) =  \frac{L_p^{(r)}(\chi, 0)}{r! L(\chi, 0) \prod_{\fp \in R'} (1 - \chi(\fp))}.
\]
The main result of this paper is a proof of the Gross--Stark Conjecture (Conjecture~\ref{c:gross}):
\begin{ithm} \label{t:main} We have $\sL_{\an}(\chi) = \sR_p(\chi)$.
\end{ithm}
 In view of (\ref{e:ineq}) and Theorem~\ref{t:main}, it now follows unconditionally that Conjecture~\ref{c:order} is equivalent to $\sR_p(\chi) \neq 0$.  This fact is known for $r \le 1$ (see \cite[Prop.~2.13]{gross}).
 
 Theorem~\ref{t:main} was proved in the case $r=1$ under certain assumptions by the first author in joint work with H.~Darmon and R.~Pollack \cite{ddp}.  These assumptions were later removed by the third author \cite{v}.  At the time of publication of \cite{ddp}, the first author believed the higher rank case to be unapproachable using the methods of
{\em loc.~cit.}  In the remainder of this introduction, we present a  detailed summary of the proof of Theorem~\ref{t:main}, highlighting the obstacles that appear when trying to generalize from $r=1$ and  describing the techniques used to overcome them.

  \begin{remark}  The fact that the endomorphism $\ell_p^\chi \circ (o_p^\chi)^{-1}$ of $U_\chi$ is canonically defined suggests the possibility that one can study its characteristic polynomial and not just its determinant.  In \cite{ds}, the first author and M.~Spiess state a conjectural formula for this characteristic polynomial in terms of the Eisenstein cocycle, generalizing the Gross--Stark Conjecture.  This more general conjecture remains open.
 \end{remark}
 
It is a pleasure to acknowledge the encouragement and suggestions of a  number of colleagues with whom we have discussed this problem over the last decade.  We are extremely grateful to J\"oel Bellaiche, David Burns, Pierre Charollois, Henri Darmon, Matthew Emerton, Ralph Greenberg,  Haruzo Hida, Chandrashekhar Khare, Masato Kurihara, Robert Pollack, Cristian Popescu, and Michael Spiess   for their advice and support.

\subsection{Explicit Formula for the Regulator}

 As noted above, we have $\dim_{E} U_{\chi} = r. $
Let $u_1, \dotsc, u_{r}$ be an $E$-basis for $U_\chi$. Write $R = \{\fp_1, \dotsc, \fp_r \}$.
For each $\fp_i \in R$,  consider the continuous homomorphisms
\begin{align*}
o_i= \ord_{\fp_i} \colon& \ F_{\fp_i}^* \longrightarrow \Z, \\
\ell_i= \log_p \circ  \Norm_{F_{\fp_i}/\Q_p} \colon& \ F_{\fp_i}^* \longrightarrow \Z_p.
\end{align*}
For each $\fp_i \in R$ choose a prime $\fP_i$ of $H$ lying above $\fp_i$.
Then via
\begin{equation} \label{e:localembed}
\cO_H[1/p] \subset H \subset H_{\fP_i} \cong F_{\fp_i},
\end{equation}
we can evaluate $o_i$ and $\ell_i$ on elements of $\cO_H[1/p]^*$, and extend by
linearity to maps
\[
o_i, \ell_i: \cO_H[1/p]^* \otimes E \longrightarrow E.
\]
Gross's regulator
is equal to the following ratio of determinants:
\begin{equation} \label{e:rpdef}
 \sR_p(\chi) = \frac{\det(-\ell_i(u_j))_{i,j = 1 \dots r}}{\det(o_i(u_j))_{i,j=1 \dots r}} \in E.
\end{equation}
It is clear that this ratio is independent of the chosen basis $\{u_i\}$. Furthermore, the ratio is independent of
the choice of $\fP_i$ since replacing $\fP_i$ by $\sigma(\fP_i)$ has the effect of scaling the $i$th row of both matrices in (\ref{e:rpdef})  by $\chi(\sigma)$.
Finally, one sees that $\det(o_i(u_j)) \neq 0$ since
the Dirichlet unit theorem implies that the $\chi^{-1}$-component of the group of $\fp_i$-units of $H$ is 1-dimensional for each $\fp_i \in R$, and hence for the appropriate basis $\{u_i\}$ the matrix $(o_i(u_j))$ can be made to equal the identity. 

\subsection{Cohomological Study of the Conjecture}   For each place $v$ of $F$, choose a decomposition group $G_{v} \subset G_F$ and let $I_v \subset G_v$ be the associated inertia group. This choice corresponds to an embedding $\overline{F} \subset \overline{F}_v$ for each place $v$ and in particular specifies a prime of $H \subset \overline{F}$ above $v$.  We assume in the sequel that the specified prime above $\fp_i$ for $\fp_i \in R$ is equal to the prime $\fP_i$ used in (\ref{e:localembed}).

Let \[ H^1_{R}(G_F, E(\chi^{-1})) \subset H^1(G_F, E(\chi^{-1})) \]  denote the subspace of continuous Galois cohomology classes $\kappa$
  unramified outside $R$, i.e.\ those classes $\kappa$ such that $\res_{I_v} \kappa \in H^1(I_v, E(\chi^{-1}))$ is trivial for all $v \not \in R$.
Note that for each prime $\fp_i \in R$ we have $\chi(G_{\fp_i}) = 1$ and hence
\[
H^1(G_{\fp_i}, E(\chi^{-1})) = H^1(G_{\fp_i}, E) = \Hom_{\cts}(G_{\fp_i}, E) \cong \Hom_{\cts}(\widehat{F_{\fp_i}^*} ,E),
\]
where the last isomorphism invokes the reciprocity isomorphism of local class field theory\footnote{Throughout this article, we adopt Serre's conventions \cite{ser} for the local reciprocity map. Therefore, if $u \in \cO_{F_\fp}^*$, then $\epsilon_{\cyc}(\rec(u)) = \Norm_{\cO_{F_\fp}/\Z_p} u$, where $\epsilon_{\cyc}$ is the usual cyclotomic character defined in (\ref{e:cyclotomic}), and $\rec(\varpi^{-1})$ is a lifting to $G_{\fp}^{\ab}$ of the Frobenius element on the maximal unramified extension of $F_{\fp}$ if $\varpi \in F_{\fp}^*$ is a uniformizer.}

\begin{equation} \rec_{\fp_i} \colon  \widehat{F_{\fp_i}^*} \longrightarrow G_{\fp_i}^{\ab}. \label{e:rec} \end{equation}
Define the subspace of ``cyclotomic classes"
\[
H^1_{\cyc}(\chi) \subset H^1_R(G_F, E(\chi^{-1}))
\]
to be the set of $\kappa$ such that for $\fp_i \in R$, the restriction $\res_{\fp_i} \kappa \in H^1(G_{\fp_i}, E)$ lies in the $E$-span of $o_i$ and $\ell_i$, viewing these as continuous homomorphisms $\widehat{F_{\fp_i}^*} \longrightarrow E.$
 Then $\dim_E H^1_{\cyc}(\chi) = r$ (this is a straightforward generalization of \cite[Lemma 1.5]{ddp}).  Let $\kappa_1, \dotsc, \kappa_r$ be a basis, and for each $\fp_j \in R$ write
\[
\res_{\fp_j} \kappa_{i} = x_{ij} o_j + y_{ij} \ell_{j},
\]
where $x_{ij}, y_{ij} \in E$. Inspired by R.~Greenberg's study of exceptional zeroes \cite{greenberg}, we define
\[
\sL_{\alg}(\chi) =  \frac{\det(x_{ij})_{i,j=1\dots r}}{\det(y_{ij})_{i,j=1\dots r}}.
\]
Using the above mentioned generalization and the fact that $\kappa_1, \dotsc, \kappa_r$ are linearly independent, it can be shown that $\det(y_{ij})_{i,j=1 \dots r} \neq 0$.

We now relate this algebraic $\sL$-invariant to the unit group $U_\chi$. Let $\kappa \in H^1_R(G_F, E(\chi^{-1}))$. Extending by $E$-linearity, we can view $\res_{\fp_i} \kappa$ as a continuous homomorphism
 \[ \res_{\fp_i} \kappa \colon \widehat{F_{\fp_i}^*} \otimes E \longrightarrow E. \]  In \S\ref{s:orthogonal}, we prove the following orthogonality result regarding $H^1_R(G_F, E(\chi^{-1}))$
and $U_\chi$.

\begin{iprop} \label{p:orth}
Let $\kappa \in H^1_R(G_F, E(\chi^{-1}))$ and $u \in U_\chi$.   Viewing $u$ as an element of $\widehat{F_{\fp_i}^*} \otimes E$ via (\ref{e:localembed}), we have
\begin{equation}
 \sum_{i=1}^{r} (\res_{\fp_i} \kappa)(u) = 0.  \label{e:explicito}
 \end{equation}
\end{iprop}

 Using Proposition~\ref{p:orth}, one readily proves that
  \[ \sL_{\alg}(\chi)   = \sR_p(\chi). \]
When $r=1$ (say $R = \{\fp\}$), Conjecture~\ref{c:gross} is therefore equivalent to the existence of a nonzero class $\kappa \in H^1_{\cyc}(\chi)$
such that $\res_\fp \kappa = \sL_{\an}(\chi)  o_\fp + \ell_\fp$. The construction of such a class is carried out in \cite{ddp} and \cite{v}.  The natural generalization of this strategy for $r > 1$ is to construct $r$ linearly independent classes in $H^1_{\cyc}(\chi)$ and to use them to compute
$\sL_{\alg}(\chi)$.  However, despite much effort, we do not in fact know how to construct even a single cyclotomic cohomology class in the general case.  The construction for $r=1$ relies crucially on the injectivity of the local restriction
\begin{equation} \label{e:restrict}
 H^1_R(G_F, E(\chi^{-1})) \longrightarrow H^1(G_\fp, E)
\end{equation}
when $R = \{\fp\}$, which in general fails for fixed $\fp \in R$ if $r > 1$.

As described below, in the general case we are still able to construct a class \[ \kappa \in H^1_R(G_F, \overline{B}(\chi^{-1})) \] for some
$E$-vector space $\overline{B}$ with partial knowledge about the local restrictions $\res_{\fp_i} \kappa$.  Our method of proof involves abandoning the hope of constructing cyclotomic classes and calculating $\sL_{\alg}(\chi)$.
 Instead, we  directly use the orthogonality (\ref{e:explicito}) with $\kappa$ and a  basis of
$U_\chi$. We describe below how the resulting equations can be used to prove that  $\sL_{\an}(\chi) = \sR_p(\chi)$.  First we describe  the mechanism through which the analytic $\sL$-invariant $\sL_{\an}(\chi)$ appears in our work and the construction of the cohomology class $\kappa$.

\bigskip

\subsection{An Infinitesimal Eigenform}
Our technique for constructing a cohomology class related to $p$-adic $L$-functions is Ribet's method, which first appeared in \cite{ribet} and was later used to great effect by Mazur and Wiles to prove the Main Conjecture of Iwasawa theory \cite{mw}, \cite{wiles}.
We consider the space of cuspidal Hida families of Hilbert modular forms for $F$ with tame level $\fn = \cond(\chi)$, and let $\T$ denote its Hecke algebra over $\Lambda  = \cO_E[[T]] $.

In \cite{ddp}, a certain linear combination of products of Eisenstein series was used to construct a cuspidal Hida Family $\sF$ that specializes in weight 1 to the Eisenstein series $E_1(1, \chi_S)$.  This Eisenstein series   is the stabilization at all primes $\fp$ above $p$ of the classical weight 1 form $E_1(1, \chi)$ with $U_\fp$-eigenvalue equal to 1.
In the case $r=1$ considered in {\em loc.\ cit.}, the form $\sF$ remains an eigenform in an infinitesimal neighborhood of weight $1$, yielding a $\Lambda$-algebra homomorphism
\begin{align}
\varphi\colon \T & \longrightarrow E[T]/T^2 \label{e:phi1} \\
t & \longmapsto a_1(t \cdot \cF) \pmod{T^2}. \nonumber
\end{align}
Here we normalize our conventions so that $T=0$ corresponds to weight $k=1$.  The explicit nature of the construction of $\sF$ allows us to calculate
\begin{equation} \label{e:phitl}
 \varphi(T_\fl) = 1 + \chi(\fl) \log_p \langle \N\fl \rangle T
 \end{equation}
for primes $\fl$ of $F$ such that $\fl \nmid \fn p$. (Here and throughout, $\langle x \rangle = x/\omega(x)$ for $x \in \Z_p^*$.)
The $p$-adic $L$-function $L_p(\chi\omega, 1-k)$ occurs as the constant term of one of the Eisenstein series used in the construction of $\sF$, and as a result an explicit computation shows that
\begin{equation} \label{e:phiup}
\varphi(U_\fp) = 1 + \sL_{\an}(\chi)T.
\end{equation}  (Equations (\ref{e:phitl}) and (\ref{e:phiup}) hold if $R'$ is nonempty; if
$R'$ is empty then slightly modified equations hold.)

In the general case, it is natural to attempt to construct a $\Lambda$-algebra homomorphism $\T \longrightarrow E[T]/T^{r+1}$ analogous to (\ref{e:phi1}). However, the form  $\cF$ constructed in \cite{ddp} is not an eigenform modulo $T^{r+1}$, and it is unclear if  the construction can be modified to define such an eigenform.  The key idea to circumvent this problem, drawn from \cite{v}, is to simply study the Hecke orbit of the form $\sF$.  Modulo $T^{r+1}$, this orbit is not 1-dimensional over $\Lambda/T^{r+1}$, but it is still finite dimensional and explicitly computable.  Therefore we obtain a representation of $\T$ into a finite dimensional $E$-algebra, namely the endomorphism ring over $E$ of the space of Fourier expansions modulo $T^{r+1}$ of the forms in the Hecke orbit of $\sF$.  These arguments are explained in detail in \S\ref{s:hhh}, culminating with the proof of the following theorem and its generalizations needed to handle all cases.  Let  $\epsilon \colon G_F \longrightarrow \Lambda^*$ denote the $\Lambda$-adic cyclotomic character (see (\ref{e:lambdaeps}) below).

\begin{ithm}   \label{t:hom}
  Suppose $R'$ is nonempty and write $R = \{ \fp_1, \dotsc, \fp_r\}$.  There exists a $\Lambda$-algebra homomorphism
  \[ \varphi \colon \T \longrightarrow W = E[T, \epsilon_1, \dotsc, \epsilon_r]/(T^{r+1}, \epsilon_i^2, \epsilon_iT, \epsilon_1\epsilon_2 \cdots \epsilon_r + (-1)^{r} \sL_{\an}(\chi) T^r) \]
  such that $T_\fl \mapsto 1 + \chi \epsilon(\fl)$ for $\fl \nmid \fn p$, $U_\fl \mapsto 1$ for $\fl \mid \fn$ or $\fl \in R'$, and
  $U_{\fp_i} \mapsto 1 + \epsilon_i$.  \end{ithm}
If $R'$ is empty, we construct a slightly more complicated homomorphism.  Note that $W$ is a local ring with maximal ideal
$\fm_W = (T, \epsilon_1, \dots, \epsilon_r)$.

\bigskip

\subsection{Construction of a Cohomology Class}
Let $\fm \subset \T$ denote the kernel of the composition of $\varphi$ with the canonical projection \[ W  \longrightarrow W/\fm_W \cong E. \]
  Let $L = \Frac(\T_{(\fm)})$ denote the total ring of fractions of the localization of $\T$ at the prime ideal $\fm$.
Theorems of Wiles and Hida imply the existence of a continuous irreducible Galois representation
\begin{align*}
\rho \colon G_F & \longrightarrow \GL_2(L)  \\
\sigma &\mapsto \mat{a(\sigma)}{b(\sigma)}{c(\sigma)}{d(\sigma)}
\end{align*}
that is unramified outside $\fn p$ and such that for primes $\fl \nmid \fn p$, the characteristic polynomial
of $\rho(\Frob_\fl)$ is
\begin{equation} \label{e:charrho1}
\chr(\rho(\Frob_\fl))(x) = x^2 - \overline{T}_\fl x + \chi \epsilon (\fl),
\end{equation}
where $\overline{T}_\fl$ denotes the image of $T_\fl$ in $L$.

Let $B$ denote the $\T$-module generated by the $b(\sigma)$.   Using the fact that $\varphi(T_\fl) = 1 + \chi\epsilon(\fl)$ together with (\ref{e:charrho1}), we show that after choosing an appropriate basis for $\rho$ the map
\[ \kappa \colon G_F \longrightarrow \overline{B} = B/\fm B \]
given by $\kappa(\sigma) = \overline{b}(\sigma) \cdot \chi^{-1}(\sigma)$ is a cocycle yielding a cohomology class in $H^1(G_F, \overline{B}(\chi^{-1}))$.
For all $\fq \mid p$, the representation $\rho|_{G_\fq}$  is known to be reducible with a certain specified semi-simplification.
This can be used to show that $\kappa$ is unramified outside $R$.

In the case $r=1$,  the injectivity of the restriction map (\ref{e:restrict})
can be used to show that after rescaling by a certain element of $L$, we have $B \subset \fm$.
Applying the homomorphism $\varphi$ to the cocycle $\kappa$  yields a class
$\kappa_\varphi \in H^1_{\fp}(G_F, E(\chi^{-1})).$  The known shape of the local representation $\rho|_{G_\fp}$ can be used to prove  that $\kappa_\varphi$ is cyclotomic.  Using equation (\ref{e:phiup}), one shows that $\res_\fp \kappa = \sL_{\an}(\chi) \cdot o_\fp +   \ell_\fp$, giving the desired result $\sL_{\alg}(\chi) = \sL_{\an}(\chi)$.

In the case $r > 1$, there is an unknown  constant $x_i \in L$ for each place $\fp_i$ such that we have a formula for the restriction of the function
$x_i b(\sigma)$ to $G_{\fp_i}$. In particular we can show that $x_i b(G_{\fp_i}) \subset \fm$.  However, the failure of the injectivity of (\ref{e:restrict}) appears to make it impossible to deduce
that $x_i B \subset \fm$.  In fact, for $r \ge 3$, we believe that this is false.\footnote{If $r=2$ and $F_{\fp_i} \cong \Q_p$ for $i = 1, 2$, then the injectivity of (\ref{e:restrict}) does hold, and one can give a proof of Theorem~\ref{t:main} in this  special case using Theorem~\ref{t:hom} and methods analogous to those of \cite{ddp}.}  In particular, we are unable to show that the cohomology class $\kappa$ is cyclotomic.

As mentioned above, our new method is to apply the orthogonality (\ref{e:orth}) with $\kappa$ and a basis $\{u_i\}$ of
$U_\chi$. We obtain $r$ equations
\[  \sum_{j=1}^{r} (\res_{\fp_j} \kappa)(u_i) =0 \]
in $\overline{B}.$
This implies that \begin{equation} \label{e:detka}
 \det(  (\res_{\fp_j} \kappa)(u_i) ) = 0  \end{equation}
in $B_R / \fm B_R$ since it is the determinant of a matrix whose rows all sum to 0, where $B_R$ is the $\T$-module generated by products
$b(\sigma_1) \cdots b(\sigma_r)$ with $\sigma_i \in G_{\fp_i}$. The fact that $x_i b(G_{\fp_i}) \subset \fm$ implies that
$(\prod_{i=1}^r x_i ) B_R \subset \fm^r$, and scaling (\ref{e:detka}) by $\prod_{i=1}^r x_i$ yields an equation in $\fm^r/\fm^{r+1}$.  We can apply
the homomorphism $\varphi$ to this equation to yield a formula in the 1-dimensional $E$-vector space $\fm_W^r = E \cdot T^r$.  An explicit computation shows that this equality is
\begin{equation} \label{e:explicitdet}
  (-1)^{r+1} \sL_{\an} \cdot \det(o_i(u_j)) + \det(\ell_i(u_j))  = 0.
\end{equation}
Equation (\ref{e:explicitdet}) is  equivalent to the desired result $\sL_{\an}(\chi) = \sR_p(\chi)$.

\section{Orthogonality Between Cohomology and Units} \label{s:orthogonal}

Recall that $H^1_R(G_F, E(\chi^{-1}))$ denotes the group of cohomology classes unramified outside $R$.  We begin by proving Proposition~\ref{p:orth} stated in the introduction.

\begin{prop}  \label{p:orth2}
Let $\kappa \in H^1_R(G_F, E(\chi^{-1}))$ and $u \in U_\chi$.   We have
\[
 \sum_{i=1}^{r} (\res_{\fp_i} \kappa)(u) = 0.
\]
\end{prop}

We will provide two proofs.  The first is more conceptual and invokes Poitou--Tate duality and the Kummer isomorphism, though we state without proof certain identifications that are needed.  The second proof is rather more direct and relies only on class field theory.

\begin{proof}[Proof 1 of Proposition~\ref{p:orth2}]

As explained in  \cite[Prop.~1.4]{ddp}, Hilbert's Theorem 90 yields isomorphisms\footnote{In (\ref{e:globalkummer}), $H^* \hat{\otimes} E = 
(\varprojlim (H^* \otimes \cO_E/p^n)) \otimes_{\cO_E} E$, and similarly in (\ref{e:localkummer}).}
\begin{align}
\delta \colon  (H^* \hat{\otimes} E)^{\chi^{-1}} &\cong H^1(G_F, E(\chi)(1)), \label{e:globalkummer}  \\
\delta_v \colon  (H_w^* \hat{\otimes} E)^{\chi^{-1}} &\cong H^1(G_v, E(\chi)(1)) . \label{e:localkummer}
\end{align}
Define
\[
H^1_{R}(G_F, E(\chi)(1)) \subset H^1(G_F, E(\chi)(1))
\]
to be the subspace of classes $\kappa$ such that $\res_v \kappa \in H^1(G_v, E(\chi)(1))$ lies in the image of $(\cO_{H, w}^* \hat{\otimes} E)^{\chi^{-1}}$ under $\delta_v$ for each $v \not\in R$.  It is then clear that (\ref{e:globalkummer}) induces an isomorphism
\begin{equation} \label{e:kummer}
\delta \colon U_\chi \cong H^1_R(G_F, E(\chi)(1)).
\end{equation}

Recall that for each place $\fp_i \in R$ there is a perfect Tate duality pairing
  \[ \langle \ , \  \rangle_{\fp_i}  \colon \xymatrix{ H^1(G_{\fp_i}, E(\chi)(1)) \times H^1(G_{\fp_i}, E(\chi^{-1})) \ar[r] & E.}
  \]
Poitou--Tate duality implies that the images of $H^1_R(G_F, E(\chi^{-1}))$  and $H^1_R(G_F, E(\chi)(1))$
under the product of the restriction maps $\res_{\fp_i}$   are orthogonal   under the local Tate duality map
\begin{equation} \label{e:orth}
 \xymatrix{ \langle \ , \ \rangle_R \colon \prod_{i=1}^{r} H^1(G_{\fp_i}, E(\chi^{-1})) \times \prod_{i=1}^{r} H^1(G_{\fp_i}, E(\chi)(1)) \ar^(.82){\sum \langle \ , \  \rangle_{\fp_i} }[rr]& & E.}
\end{equation}
The desired result follows from this orthogonality and the fact that \begin{equation} \label{e:texplicit}
 \langle \kappa, \delta(u) \rangle_{\fp_i} = (\res_{\fp_i} \kappa)(u). \end{equation}
 \end{proof}

We now present an alternate and more direct proof of (\ref{e:explicito}) using only general facts from class field theory.

\begin{proof}[Proof 2 of Proposition~\ref{p:orth2}]
Since $H$ is the fixed field of $\chi$, the restriction of $\kappa$ to $G_H$ yields a class
\[ \res_H \kappa \in H^1(G_H, E(\chi^{-1}))^G = \Hom_{\cts}(G_H, E)^{G = \chi^{-1}}, \]
where the group on the right is the $E$-vector space of continuous group homomorphisms
$f\colon G_H \rightarrow E$ such that
 \begin{equation} \label{e:inv}
 f(\sigma h \sigma^{-1}) = \chi^{-1}(\sigma) f(h) \qquad \text{for } \sigma \in G_F, h \in G_H.
 \end{equation}
   Since $\kappa$ is unramified outside $R$, the homomorphism
$\res_H \kappa$ is trivial on the inertia group $I_w \subset I_v$ for each place $v \not \in R$, where $w$ is the place of $H$ specified by the choice of $G_v$.
From (\ref{e:inv}), it follows that $\res_H \kappa$ is trivial on the inertia group $I_w$ for every place $w \not \in R_H$, where $R_H$ denotes the set of places of $H$ lying above those in $R$.  Therefore the homomorphism $\res_H \kappa$ factors through the maximal abelian extension of $H$ unramified outside $R_H$, which we denote by $K$.
By class field theory, we have an isomorphism
\begin{equation} \label{e:galkh}
 \rec \colon \A_H^* / \overline{H^* \prod_{w \not \in R_H} \cO_{H,w}^*} \longrightarrow \Gal(K/H),
 \end{equation}
 where $\A_H$ is the ring of adeles of $H$ and by convention $\cO_{H,w}^* = \C^*$ if $w$ is a complex place.

Let $u \in \cO_{R_H}^*$, the group of $R_H$-units of $H$.  The id\`ele
\[ \pi_u =(u, u, \dotsc, u, 1, 1, \dotsc) \]
with component 1 at each $w \not \in R_H$ and component $u$ at each $w \in R_H$ is clearly trivial in the quotient (\ref{e:galkh}).
The fact that $\res_H \kappa$ factors through $\Gal(K/H)$ therefore implies that
\[  0 = (\res_H \kappa)(1) = (\res_H \kappa)(\rec(\pi_u)) = \sum_{w \in R_H} (\res_w \kappa)(u) = \sum_{i=1}^{r} \sum_{\sigma \in G} (\res_{\sigma(\fP_i)} \kappa )(u). \]
Equation (\ref{e:inv}) implies that
\[  \res_{\sigma(\fP_i)}(u) = \chi^{-1}(\sigma) \res_{\fP_i}(\sigma^{-1}(u)),\] and noting that via (\ref{e:localembed})
we have $\res_{\fP_i} = \res_{\fp_i}$,
we obtain
\[  \sum_{i=1}^{r} (\res_{\fp_i} \kappa)(u_\chi) = 0 \qquad \text{where } u_\chi = \sum_{\sigma \in G} \sigma(u) \otimes \chi(\sigma). \]
Since  elements of the form $u_\chi$ for $u \in \cO_{R_H}^*$ generate the $E$-vector space $U_\chi$, the desired result follows.
\end{proof}

We conclude this section by proving a crucial injectivity result from global to local cohomology groups.

\begin{prop}  \label{p:unr} The restriction map
\[ \prod_{i=1}^r \res_{I_{\fp_i}} \colon H^1_R(G_F, E(\chi^{-1})) \longrightarrow \prod_{i=1}^r H^1(I_{\fp_i}, E) \]
is injective.
\end{prop}

As mentioned in the introduction, the fact that in the general case this injectivity fails to hold when
$\prod_{i=1}^r \res_{I_{\fp_i}}$ is replaced by a single $\res_{I_{\fp_i}}$ (or even a single $\res_{\fp_i}$) represents
an important distinction from the rank 1 setting.

\begin{proof}
The proposition states that there are no nonzero classes in $H^1(G_F, E(\chi^{-1}))$ that are unramified everywhere.
To see this, first note that the restriction map
\[ \res_H \colon H^1(G_F, E(\chi^{-1})) \longrightarrow H^1(G_H, E)^{G = \chi^{-1}} \]
is an isomorphism, since the preceding and following terms in the inflation-restriction exact sequence  are the groups $H^i(G, E(\chi^{-1}))$ for $i=1, 2$.  These groups vanish
since $G= \Gal(H/F)$ is finite and the $E$-vector space $E(\chi^{-1})$ is torsion-free.

 If $\kappa$ is unramified everywhere, then as in the second proof of Proposition~\ref{p:orth2} we see that $\res_H \kappa$ factors through the maximal abelian unramified extension of $H$.  Since this extension (the Hilbert class field of $H$) is a finite extension of $H$, it follows that $\res_H \kappa =0$ once again using the fact that $E$ is torsion-free.  The fact that $\res_H$ is an isomorphism then implies that $\kappa = 0$ as desired.
\end{proof}

\section{Homomorphism on the Hida Hecke Algebra} \label{s:hhh}

Our goal in this section is to prove Theorem~\ref{t:hom} from the introduction and its various generalizations that are needed to handle all cases.  This involves rather technical computations involving the Hecke action on certain explicitly defined Hida families.  The reader who is willing to take Theorem~\ref{t:hom} as a black box and is interested in the deduction of the equality $ \sL_{\an}(\chi) = \sR_p(\chi) $ from this theorem can skip ahead to \S \ref{s:ccc} without any loss of continuity.

We first recall the notation and conventions of \cite[\S2 and \S3]{ddp} and \cite{v} for Hida families of Hilbert modular forms for $F$.

\subsection{Notation on Hida Families}

Let  $\Lambda = \cO_E[[T]]$ where, as in the introduction,  $E$ is a finite extension of $\Q_p$ containing the values of the character $\chi$.
 For each  $k \in \Z_p$ we have a ``specialization to weight $k$" $\cO_E$-algebra homomorphism
 \begin{equation} \label{e:nukdef}
  \nu_k\colon \Lambda \longrightarrow \cO_E \text{ given by  } T \mapsto u^{k-1} - 1,
  \end{equation}
   where $u$ is a topological generator of $1 + 2p\Z_p$ (for instance, we may choose $u = 1+p$ if $p$ is odd and $u=5$ if $p=2$). Under this convention, specialization to weight $1$ corresponds to the  augmentation map $T \mapsto 0$.  Let $\Lambda_{(1)} = \cO_E[[T]]_{(T)}$ denote the localization of $\Lambda$ in weight 1, i.e.\ the localization of $\Lambda$ with respect to the prime ideal $(T) = \ker \nu_1$.  Note that $p$ is invertible in $\Lambda_{(1)}$, so in particular $\Lambda_{(1)}$ is an $E$-algebra.  Furthermore $\Lambda_{(1)}$ is a DVR and we choose the uniformizer \[ \pi = \frac{1}{\log_p u} T. \]  This uniformizer is normalized to have the following property making translation between the $k$-variable and the $\pi$-variable straightforward.  Suppose $h \in \Lambda_{(1)}$ can be written $h = \pi^n h'$ where $h' \in \Lambda_{(1)}^*$, and let $f \colon U \rightarrow E$ be defined for a sufficiently small neighborhood $U \subset \Z_p$ containing $1$ by $f(k) = \nu_k(h)$.  Then $f$ has a zero of order $n$ at $k=1$ and \[ f^{(n)}(1)/n! = \nu_1(h'). \]

   Next we recall the $\Lambda$-adic cyclotomic character.  This is the character $\epsilon \colon G_F \rightarrow \Lambda^*$ satisfying
  $\nu_k(\epsilon(\sigma)) = \langle\epsilon_{\cyc}(\sigma)\rangle^{k-1}$ for any $k \in \Z_p$.  Here \begin{equation} \label{e:cyclotomic}
  \epsilon_{\cyc}\colon G_F \longrightarrow \Z_p^* \end{equation}  is the usual cyclotomic character defined by $\sigma(\zeta) = \zeta^{\epsilon_{\cyc}(\sigma)}$ for any $p$-power root of unity $\zeta$.  The character $\epsilon$ is given explicitly by the formula
  \begin{equation} \label{e:lambdaeps}
   \epsilon(\sigma) = (1 + T)^{\log_p \langle \epsilon_{\cyc}(\sigma) \rangle/\log_p u}. \end{equation}

Recall that $\fn$ denotes the conductor of the character $\chi$.  We denote by $\cM(\fn, \chi)$ the $\Lambda$-module of $\Lambda$-adic Hilbert modular forms for $F$ with tame level $\fn$ and  character $\chi$.  For each $\sF \in \cM(\fn, \chi)$ and integer $k \ge 2$, the specialization $\nu_k(\sF)$ lies in the space $M_k(\fn p, \chi\omega^{1-k})$ of Hilbert modular forms for $F$ of weight $k$, level $\fn p$, and character $\chi\omega^{1-k}$.  The subspace of cusp forms in $\cM(\fn, \chi)$ is denoted $\cS(\fn, \chi)$.
The $\Lambda$-module $\cM(\fn, \chi)$ is equipped with an action of Hecke operators $T_\fl$ for primes $\fl \nmid \fn p$ and $U_\fl$ for $\fl \mid p$.  Following Hida, we let
\[ e = \lim_{n \rightarrow \infty} \left(\prod_{\fp \mid p} U_\fp \right)^{n!} \] be the ordinary projector and denote by
\[ \cM^o(\fn, \chi) = e \cM(\fn, \chi), \qquad \cS^o(\fn, \chi) = e \cS(\fn, \chi) \]
the spaces of Hida families and cuspidal Hida families, respectively.  We denote by $\tilde{\T}$ and $\T$ the $\Lambda$-algebras of Hecke operators acting on
$ \cM^o(\fn, \chi)$ and $ \cS^o(\fn, \chi)$, respectively.

Of particular interest to us will be the Eisenstein series.  Let $k \ge 1$ be an integer and let $\eta$  be a narrow ray class character of $F$ such that
$\eta$ is totally odd or totally even, with parity agreeing with $k$.  Let $\fb$ denote the modulus of $\eta$, which we do not assume to equal the conductor of $\eta$ (i.e.\ $\eta$ need not be a primitive character).  Excluding the exceptional case where $F=\Q, k=2,$ and $\fb = 1$, there is an Eisenstein  series $E_k(1, \eta)$ with normalized Fourier coefficients given by
\[ c(\fa, E_k(1, \eta)) = \sum_{\fr \mid \fa, (\fr, \fb)=1} \eta(\fr)  \N \fr^{k-1} \]
for integral ideals $\fa \subset \cO_F$ and constant coefficients (assuming $\fb \neq 1$ or $k \neq 1$)\footnote{If  $\fb = 1$ and $k=1$, the constant coeffcients are given by
\[ c_\lambda(0, E_1(1, \eta)) = 2^{-[F:\Q]}(L(\chi, 0) + \chi^{-1}(\lambda)L(\chi^{-1}, 0)). \]}
\[ c_\lambda(0, E_k(1, \eta)) = 2^{-[F:\Q]}L_\fb(\eta, 1-k), \qquad \lambda \in \Cl^+(F), \]
where the subscript $\fb$ emphasizes that the Euler factors at primes dividing $\fb$ are removed.
(For details regarding our conventions on Hilbert modular forms and their Fourier coefficients, see \cite[\S2]{ddp}.)
These classical Hilbert modular forms interpolate $p$-adically in the sense that there is an Eisenstein series
 $\sE(1, \chi) \in  \cM^o(\fn, \chi)$ such that $\nu_k(\sE(1, \chi)) = E_k(1, \chi \omega^{1-k})$ for all $k \ge 1$, where
 the character $\chi\omega^{1-k}$ is understood to always have modulus divisible by all primes above $p$ (even if $k \equiv 1\pmod{p-1}$).
The constant coefficients of $\nu_k(\sE(1, \chi))$ can be expressed as $2^{-[F:\Q]} L_p(\chi\omega, 1-k)$.

\subsection{Construction of a Cusp Form} \label{s:ccf}

We now recall the construction of a certain Hida family of cusp forms from \cite{ddp} and \cite{v}.  For any integer $k$,
we let $\Lambda_{(k)} = \Lambda_{(T - u^{k-1} - 1)}$ denote the localization of $\Lambda$ in weight k, i.e.\ the
localization at the prime ideal $(T - u^{k-1} - 1) = \ker \nu_k$.  Similarly we let $\cM^o(1, \omega^{-1})_{(k)}$
denote the localization of the space of Hida families of modular forms with respect to weight $k$, i.e.\
\[ \cM^o(1, \omega^{-1})_{(k)} = \cM^o(1, \omega^{-1}) \otimes_{\Lambda} \Lambda_{(k)}. \]

\begin{lemma}[\cite{v}, Theorem 2]  \label{l:kevin}
There exists a Hida family $\sG \in \cM^o(1, \omega^{-1})_{(0)}$ with the property that $\nu_0(\sG) = 1$ and
$c_\lambda(0, \sG) = 1$ for all $\lambda \in \Cl^+(F)$.
\end{lemma}

Lemma~\ref{l:kevin} was proved in \cite{ddp} under the assumption of Leopoldt's conjecture using Eisenstein series, but it was demonstrated unconditionally in \cite{v}.
We write \[ G_k = \nu_k(\sG) \in M_k(p, \omega^{-k}). \]
Now, for each integer $k \ge 1$, we define a modular form $F_k \in M_k(\fn p, \chi \omega^{1-k})$.  If $R'$ is not empty (we call this case 1), let
 \begin{equation} \label{e:fkdef}
  F_k = E_k(1, \chi \omega^{1-k}) - E_1(1, \chi_{R'})\cdot G_{k-1} \cdot \frac{L_p(\chi\omega, 1-k)}{L(\chi_{R'}, 0)}.
  \end{equation}
Here $\chi_{R'}$ denotes the character $\chi$ viewed with modulus divisible by all primes in $R'$, so
\[ L(\chi_{R'}, 0) = L(\chi, 0) \prod_{\fp \in R'}( 1- \chi(\fp)) \]
is equal (up to the constant $2^{-[F:\Q]}$) to the value of the constant terms of $E_1(1, \chi_{R'})$.  By construction, $F_k$ has constant terms equal to 0.
If $R'$ is empty (this setting will be subdivided further into two cases, case 2 and case 3) we let
 \begin{align}
  F_k & = E_k(1, \chi\omega^{1-k}) - E_1(1, \chi)\cdot G_{k-1} \cdot \frac{L_p(\chi \omega, 1-k)}{L(\chi,0)}  \nonumber
 \\ &  \ \ \ \ \
+ E_k(\chi, \omega^{1-k}) \cdot \frac{L_p( \chi\omega, 1-k)}{L(\chi,0)} \cdot \frac{L(\chi^{-1},0)}{L_p(\chi^{-1}\omega, 1-k)}. \label{e:fkcase2}
 \end{align}
 Again $F_k$ has constant terms equal to 0.

The forms $F_k$ interpolate to Hida families.  Note that \[ \nu_k(\sG((1+T)u^{-1} - 1)) = \nu_{k-1}(\sG(T)). \]
Therefore, in case 1 the $\Lambda$-adic family
\[ \tilde{\sF} = \sE(1, \chi) - E_1(1, \chi_{R'})\sG((1+T)u^{-1} - 1)) \cdot \frac{\cL(\chi \omega)}{L(\chi_{R'}, 0)}\]
satisfies $\nu_k(\tilde{\sF}) = F_k$ for all positive integers $k$ in a neighborhood of $1$ in $\Z_p$, where $\cL(\chi \omega) \in \Lambda_{(1)}$ is the element such that $\nu_k(\cL(\chi \omega)) = L_p(\chi \omega, 1-k)$.
Similarly, if $R' = \phi$ we define
\begin{equation} \label{e:ftilde}
 \tilde{\sF} = \sE(1, \chi) - E_1(1, \chi)\cdot \sG((1+T)u^{-1} - 1)) \cdot \frac{\cL(\chi \omega)}{L(\chi,0)} + \sE(\chi, 1) \cdot \cw,
 \end{equation}
where
\[
\cw = \frac{ \cL(\chi \omega)}{\cL(\chi^{-1} \omega)} \cdot \frac{L(\chi^{-1},0)}{ L(\chi,0)}\in \Frac(\Lambda)
\]
satisfies
\begin{equation} \label{e:cwdef}
 \nu_k(\cw) = \frac{L_p( \chi\omega, 1-k)}{L(\chi,0)} \cdot \frac{L(\chi^{-1},0)}{L_p(\chi^{-1}\omega, 1-k)}
 \end{equation}
for all  $k \in \Z_p$ with $L_p(\chi^{-1}\omega, 1-k) \neq 0$. In our calculations, we will require that the $\Lambda$-adic form $\tilde{\sF}$ is regular in weight $1$, i.e. $\sF \in \cM^o(\fn, \chi)_{(1)}.$  This will be the case unless $\cW$ has a pole in weight 1, i.e. if
\[ \ord_{\pi} \cW = r_{\an}(\chi) - r_{\an}(\chi^{-1}) < 0.\]
(Of course, Conjecture~\ref{c:order} implies that $r_{\an}(\chi) = r(\chi) = r_{\chi^{-1}} = r_{\an}(\chi)$, so it should be the case that $\ord_{\pi} \cW = 0$; however we are proving Conjecture~\ref{c:gross} without assuming Conjecture~\ref{c:order}, so we need to consider the possibility  $\ord_{\pi} \cW < 0$.)
Now, swapping $\chi$ and $\chi^{-1}$ has the effect of inverting $\cW$.  Therefore, in the case that $\cW$ has a pole at $k=1$, it suffices instead to assume that $\cW$ has a zero at $k=1$ and to prove Conjecture~\ref{c:gross} for $\chi^{-1}$ (i.e. to prove that $\sL_{\an}(\chi^{-1}) = \sR_p(\chi^{-1})$).  Therefore, we assume that $\ord_\pi \cW \ge 0$ and subdivide the setting $R' = \phi$ into two cases:
\begin{itemize}
\item Case 2: $\nu_1(\cW) \neq 0$; we must prove $\sR_p(\chi) = \sL_{\an}(\chi)$.
\item Case 3: $\nu_1(\cW) = 0$; we must prove $\sR_p(\chi) = \sL_{\an}(\chi) = 0$ and $\sR_p(\chi^{-1}) = \sL_{\an}(\chi^{-1})$.
\end{itemize}

Now, the $\Lambda$-adic family of modular forms $\tilde{\sF}$ has been constructed such that its constant coefficients at $\infty$ vanish---in the terminology of \cite{ribet}, $\tilde{\sF}$ is a ``semi-cusp form."
The following result was proved in \cite[Corollary 2.10 and Proposition 3.4]{ddp}.
\begin{theorem}  There exists a Hecke operator $t \in \tilde{\T}_{(1)}$ such that
$\nu_1(t)(E_1(1, \chi_S)) = E_1(1,\chi_S)$ and such that $\sF = t \cdot e \cdot \tilde{\sF}$ is a cuspidal Hida family, i.e. $\sF \in \cS^o(\fn, \chi)_{(1)}$.
\end{theorem}

\subsection{Hecke Action in Case 1 ($R' \neq \phi$)}

We now study the action of the Hecke operators on the form $\sF$.  The action of the Hecke operators above $p$ is more complicated than
the setting $r=1$ considered in \cite{ddp}, and our methods here draw from those introduced in \cite{v}.
 Write
\[ r_{\an} = r_{\an}(\chi) = \ord_{s=0} L_p(\chi \omega, s), \qquad \sL_{\an}^*(\chi) =  \frac{L_p^{(r_\an)}(\chi, 0)}{r_\an! L(\chi, 0) \prod_{\fp \in R'} (1 - \chi(\fp))}.  \]
(Of course, Conjecture~\ref{c:order} states that $r_{\an} = r$ and hence $\sL_{\an}^*(\chi) = \sL_{\an}(\chi)$, but we are not assuming this conjecture.)

Any Hida family is determined by its Fourier expansion; there is a canonical $\Lambda$-algebra embedding
\[ c\colon \cS^o(\fn, \chi)_{(1)} \longrightarrow \prod_{\fa \subset \cO_F} \Lambda_{(1)}, \qquad \sH \mapsto (c(\fa, \sH))_{\fa \subset \cO_F}. \]
We define $\cH$ to be the image of the Hecke orbit of $\sF$ under the reduction of $c$ modulo $\pi^{r_{\an}+1}$.  This is a finitely-generated
module over $\Lambda_{(1)}/\pi^{r_\an+1} = E[\pi]/\pi^{r_{\an}+1}$, and  we obtain a canonical $\Lambda$-algebra homomorphism
\begin{equation} \label{e:phidef}
 \varphi \colon \T \longrightarrow \End_{E[\pi]/\pi^{r_\an+1}} \cH. \end{equation}
By identifying the image of (\ref{e:phidef}), we can now prove Theorem~\ref{t:hom} from the introduction.
\begin{theorem}   \label{t:case1}
  Suppose $R'$ is nonempty.  There exists a $\Lambda$-algebra homomorphism
  \[ \varphi \colon \T \longrightarrow W_1 = E[\pi, \epsilon_1, \dotsc, \epsilon_r]/(\pi^{r_\an+1}, \epsilon_i^2, \epsilon_i \pi, \epsilon_1\epsilon_2 \cdots \epsilon_r + (-1)^{r_\an} \sL_{\an}^*(\chi) \pi^{r_\an}) \]
  such that
  \begin{alignat*}{2}
  T_\fl & \mapsto 1 + \chi \epsilon(\fl) \quad  &&\text{for } \fl \nmid \fn p \\
  U_\fl & \mapsto 1 \quad && \text{for } \fl \mid \fn \text{ or } \ell \in R', \text{ and } \\
  U_{\fp_i} & \mapsto 1 + \epsilon_i, && R = \{\fp_1, \dotsc, \fp_r\}.
  \end{alignat*}
    \end{theorem}

\begin{proof}

By definition,    $\pi^{r_\an}$ fully divides $\cL(\chi\omega)$ in $\Lambda_{(1)}$.  Since $\nu_0(\sG) = 1$, it follows that modulo $\pi^{r_{\an}+1}$ we can write the second term appearing in the definition of $\tilde{\sF}$ more simply, namely:
\begin{align}
\sF' &=  E_1(1, \chi_{R'})\sG((1+T)u^{-1} - 1)) \cdot \frac{\cL(\chi \omega)}{L(\chi_{R'}, 0)} \nonumber \\
&\equiv (-1)^{r_\an} E_1(1, \chi_{R'}) \sL_{\an}^*(\chi)\pi^{r_\an} \pmod{\pi^{r_\an+1}}. \label{e:fpcong}
\end{align}
To be clear, this congruence means that the two sides have Fourier coefficients that are congruent modulo $\pi^{r_\an+1}$.
In particular, modulo $\pi^{r_\an+1}$ the Hecke action on $\sF'$
depends only on the action on the form $E_1(1, \chi_{R'})$.  More precisely, if $\tau \in \tilde{\T}$
 then we have
\begin{equation} \label{e:tfp}
 \tau \sF' \equiv   (-1)^{r_\an} \nu_1(\tau)(E_1(1, \chi_{R'}))\cdot \sL_{\an}^*(\chi)\pi^{r_\an} \pmod{\pi^{r_\an+1}}.
 \end{equation}
Let us therefore study the action of the Hecke operators on $E_1(1, \chi_{R'})$.
We have
\begin{alignat*}{2}
 T_\fl  E_1(1, \chi_{R'}) &= (1 + \chi(\fl))  E_1(1, \chi_{R'}), \qquad && \fl \nmid \fn p.  \\
U_\fl  E_1(1, \chi_{R'}) &= E_1(1, \chi_{R'}),  && \fl \mid \fn \text{ or } \fl \in R'.
\end{alignat*}
The action of the operators $U_\fp$ for $\fp \in R$ is more subtle and leads to interesting phenomenon.  A direct calculation shows that for $\fp \in R$, we have
\[ U_\fp   E_1(1, \chi_{R'}) =  E_1(1, \chi_{R'})  +  E_1(1, \chi_{R' \cup \{\fp\}}). \]
More generally, for $R' \subset J \subset S_p$ and $\fp \in S_p$, we have
\begin{equation} \label{e:upaction}
 (U_\fp - 1) E_1(1, \chi_J) = \begin{cases}
E_1(1,\chi_{J \cup \{\fp\}}) & \text{ if } \fp \not \in J, \\
0 & \text{ if } \fp \in J.
\end{cases}
 \end{equation}

Note that for $\fl \nmid \fn p$, we have $T_\fl(\sE(1, \chi)) = (1 + \chi \epsilon(\fl)) \sE(1, \chi)$.
Since \[  1 + \chi \epsilon(\fl) \equiv 1+ \chi(\fl) \pmod{\pi}, \]
it follows from (\ref{e:tfp}) and the definition of $\tilde{\sF}$ that modulo $\pi^{r_\an + 1}$, the Hecke operator $T_\fl$ acts as multiplication
by the scalar $1 + \chi \epsilon(\fl)$ on $\tilde{\sF}$.  By the commutativity of the Hecke algebra, the same is clearly true for $\sF$ and its entire Hecke orbit $\cH$.
The same argument shows that $U_\fl$ for $\fl \mid \fn$ or $\fl \in R'$ acts as the identity on  $\cH$.
Therefore the homomorphism (\ref{e:phidef}) satisfies
\begin{alignat}{2}
\varphi(T_\fl) &=  1 + \chi \epsilon(\fl) \quad && \text{ for } \fl \nmid \fn p, \label{e:tlambda} \\
 \varphi(U_\fl) &= 1  &&   \text{ for } \fl \mid \fn \text{ or } \fl \in R'.
 \end{alignat}

 Recall that $R = \{\fp_1, \dotsc, \fp_r\}$.  For $\fp_i \in R$, the operator $U_{\fp_i} - 1$ annihilates $\sE(1, \chi)$.  It follows from this along with (\ref{e:fpcong}) that $\pi$ annihilates
 the image of $(U_{\fp_i} - 1)\sF$ in $\cH$.  Similarly using
 (\ref{e:tfp}) and (\ref{e:upaction}), it follows that the image of $(U_{\fp_i} - 1)^2\sF$ is $0$ in $\cH$.
If we let $\overline{\epsilon}_i$ denote the image of $U_{\fp_i} - 1$ under the homomorphism $\varphi$ given in (\ref{e:phidef}), it is therefore clear that
\begin{equation}  \label{e:epsilonpi}
\overline{\epsilon}_i^2 = 0 \qquad \text{ and  } \qquad \overline{\epsilon}_i \cdot \pi = 0 \text{ for all } i.
\end{equation}
Finally, we consider the action of $\prod_{i=1}^{r} (U_{\fp_i} - 1)$.  We have
\begin{alignat*}{2}
 \prod_{i=1}^{r} (U_{\fp_i} - 1) \sF & \equiv t \cdot e  ((-1)^{r_\an + 1} E_1(1, \chi_{S}) \sL_{\an}^*(\chi)\pi^{r_\an}) && \pmod{\pi^{r_\an + 1}} \\
& \equiv t \cdot e ((-1)^{r_\an + 1}  \sL_{\an}^*(\chi)\pi^{r_\an} \sE(1, \chi)) && \pmod{\pi^{r_\an + 1}} \\
& \equiv (-1)^{r_\an + 1}  \sL_{\an}^*(\chi)\pi^{r_\an} \sF &&  \pmod{\pi^{r_\an + 1}}.
 \end{alignat*}
Therefore we have \begin{equation} \label{e:epslan}
 \overline{\epsilon}_1 \cdots \overline{\epsilon}_r + (-1)^{r_\an}  \sL_{\an}^*(\chi)\pi^{r_\an}  = 0 \qquad \text{in } \End_{E[\pi]/\pi^{r_\an + 1}} \cH.
 \end{equation}

Combining (\ref{e:tlambda})--(\ref{e:epslan}), we have therefore proved that there is a surjective $\Lambda_{(1)}$-algebra homomorphism
\[
W_1 \longrightarrow \varphi(\T) \otimes_{\cO_E} E
\]
such that $\epsilon_i \mapsto \overline{\epsilon_i}$.  To conclude the proof, we must show that this homomorphism is injective. This can be achieved by counting dimensions.  The algebra $W_1$ has dimension $2^{r_{\an}} + {r_\an} - 1$ over $E$, and is generated as an $E$-vector space by $1, \pi, \pi^2, \dotsc, \pi^{r_\an-1}$ and the products $\prod_{j \in J} \epsilon_j$ for all subsets $J \subset R$, $J \neq \phi$.  We must therefore show that the elements $1, \pi, \pi^2, \dotsc, \pi^{r_\an - 1}$ and the products $\prod_{j \in J} \overline{\epsilon}_j$
are $E$-linearly independent in $\End_{E[\pi]/\pi^{r_\an+1}} \cH$, and for this it suffices to show that their images on $\sF$ are $E$-linearly independent.
It is clear that the coefficients of $\sF, \pi \sF, \dotsc, \pi^{r_\an - 1} \sF$ in any putative linear combination must be zero, since these forms all vanish to distinct orders less than $r_\an$ at $k=1$.  We have already calculated that up to a nonzero constant multiple, the forms $\prod_{j \in J} \overline{\epsilon}_j \sF$ for $J \neq \phi$ are congruent to
$E_1(1, \chi_{R' \cup J}) \pi^{r_\an}$ modulo $\pi^{r_{\an}+1}$.  These forms are easily seen to be linearly independent over $E$, and the result follows.
\end{proof}

\begin{remark}
For our applications, we only require the subalgebra $\T' \subset \T$ generated by the operators $T_\fl$ for $\fl \nmid \fn p$, $U_\fl$ for $\fl \mid \fn$ or $\fl \in R'$, and
\[ U_J = T^{r - \#J}\prod_{\fp \in J} (U_{\fp} - 1) \]
for nonempty subsets $J \subset R$.  Restricting  the homomorphism $\varphi$ to $\T'$ and reducing modulo $\pi^{r+1}$ (this reduction is only relevant if $r_\an > r$) yields a $\Lambda$-algebra homomorphism
\[ \varphi' \colon \T' \longrightarrow E[\pi]/\pi^{r+1} \]
satisfying
\begin{alignat*}{2}
T_\fl & \mapsto 1 + \chi \epsilon(\fl)  && \text{ for } \fl \nmid \fn p, \\
 U_\fl &\mapsto 1 && \text{ for } \fl \mid \fn \text{ or } \fl \in R', \\
  U_J &\mapsto 0   && \text{ for } \phi \neq J \subsetneq R, \\
  U_R &\mapsto (-1)^{r+1} \sL_{\an}(\chi) \pi^r.
  \end{alignat*}
  This holds even if $r_{\an} > r$, in which case $\sL_{\an}(\chi) = 0$.
The homomorphism $\varphi'$ can be constructed directly and more simply as the mod $\pi^{r+1}$-eigenvalues of the form $\sF$, i.e.\ for all $\tau \in \T'$ we have
\[ \tau \sF \equiv \varphi'(\tau) \sF \pmod{\pi^{r+1}}. \]
Even though the homomorphism $\varphi'$ is sufficient for our applications, we have included the construction of the homomorphism $\varphi$ on the full Hecke algebra $\T$ for completeness.
\end{remark}

\begin{remark} If $r = 1$, there is a natural $\Lambda_{(1)}$-algebra homomorphism $W_1 \longrightarrow E[\pi]/\pi^2$ sending $\epsilon_1 \mapsto  \sL_{\an}(\chi) \pi$.  (Note that this holds even if $r_{\an} > r = 1$, in which case $\sL_{\an}(\chi) = 0$.)  The composition of $\varphi$ with this homomorphism is precisely the homomorphism constructed in case 1 in \cite{ddp}.
\end{remark}

\subsection{Hecke Action in Case 2: $R' = \phi$, $\nu_1(\cW) \neq 0$}

In this section, we handle the more complicated setting where $R'  = \phi$.  Recall that  we are assuming that $\cW \in \Lambda_{(1)}$ so that
the family $\sF$ is regular in weight 1.
Define the $\Lambda_{(1)}$-algebra \[ W_2 = E[\pi, \epsilon_1, \dotsc, \epsilon_r, y]/I_{W_2} \]
  where
   \begin{align*}
    I_{W_2} = & (\pi^{r_\an+1}, y^{{r_\an} + 1}, y(\pi - y), \pi^{r_{\an}}\cW - y^{r_{\an}}(\cW + 1),  \\
    & \ \ \ \  \epsilon_i^2, \epsilon_i \pi, \epsilon_i y , \epsilon_1\epsilon_2 \cdots \epsilon_r + (-1)^{r_\an} \sL_{\an}^*(\chi)(\pi^{r_\an} - y^{r_{\an}})).
    \end{align*}
\begin{theorem}   \label{t:case2}
  Suppose $R'$ is empty.
  If $\nu_1(\cW) \neq 0$, then
  there exists a $\Lambda$-algebra homomorphism
  \[ \varphi \colon \T \longrightarrow W_2 \]
  such that
  \begin{alignat*}{2}
  T_\fl & \mapsto 1 + \chi \epsilon(\fl) + (\chi(\fl) - 1)\frac{1 - \epsilon(\fl)}{\pi} y\quad  &&\text{ for } \fl \nmid \fn p \\
  U_\fl & \mapsto 1  + \frac{\epsilon(\fl) - 1}{\pi} y \quad && \text{ for } \fl \mid \fn, \text{ and } \\
  U_{\fp_i} & \mapsto 1 + \epsilon_i.
  \end{alignat*}
    \end{theorem}

\begin{proof}  The proof follows that of Theorem~\ref{t:case1}.  We again let $\cH$ denote the image of the Hecke span of $\sF$ in the space of Fourier coefficients
modulo $\pi^{r_\an + 1}$, and consider the canonical $\Lambda$-algebra homomorphism
\begin{equation} \label{e:phidef2}
\varphi \colon \T \longrightarrow \End_{E[\pi]/\pi^{r_\an + 1}} \cH.
\end{equation}

Fix a prime $\fq \nmid \fn p$ such that $\chi(\fq) \neq 1$.  Define
\[ Y = \frac{T_\fq - 1- \chi \epsilon(\fq)}{(\chi(\fq) - 1)(1 - \epsilon(\fq))/\pi} \in \tilde{\T}_{(1)}. \]
An explicit computation shows that
\[ Y \tilde{\sF} \equiv \pi \sE(\chi, 1) \cW \pmod{\pi^{r_{\an}+ 1}}. \]
It therefore follows that
\begin{alignat}{3}
T_\fl \sF & \equiv \left(1 + \chi \epsilon(\fl) + (\chi(\fl) - 1)\frac{1 - \epsilon(\fl)}{\pi} Y\right) \sF  && \pmod{\pi^{r_\an}+1}, \quad && \fl \nmid \fn p \label{e:tly} \\
U_\fl \sF & \equiv \left(1  + \frac{\epsilon(\fl) - 1}{\pi} Y\right) \sF && \pmod{\pi^{r_\an}+1},  && \fl \mid \fn.
\end{alignat}
One also computes $Y \sE(\chi, 1) = \pi \sE(\chi, 1)$ and hence:
\begin{alignat}{2}
(\pi Y - Y^2) {\sF}  & \equiv 0 && \pmod{\pi^{r_{\an} + 1}}, \\
Y^{r_{\an} + 1} {\sF} & \equiv 0 && \pmod{\pi^{r_{\an} + 1}}, \\
(\pi^{r_\an} \cW- Y^{r_\an}(\cW +1)) {\sF} & \equiv 0 && \pmod{\pi^{r_{\an} + 1}}. \label{e:relation3}
\end{alignat}
In computing (\ref{e:relation3}), one uses $\sE(1, \chi) \equiv \sE(\chi, 1) \equiv E_1(1, \chi_S) \pmod{\pi}$.  Now we consider the action of the Hecke operators above $p$ on $\sF$ modulo $\pi^{r_\an+1}$.
As in Theorem~\ref{t:case1}, we have \begin{equation}
 (U_\fp - 1)^2 \sF \equiv \pi (U_\fp - 1) \sF \equiv 0 \pmod{\pi^{r_{\an} + 1}}, \end{equation}
 and clearly also \begin{equation} (U_\fp - 1) Y \sF \equiv 0 \pmod{\pi^{r_{\an} + 1}}.
 \end{equation}
 We furthermore compute:
\begin{alignat*}{2}
 \prod_{i=1}^{r} (U_{\fp} - 1) \tilde{\sF} & \equiv (-1)^{r_{\an} + 1} E_1(1, \chi_S) \sL_{\an}^*(\chi) \pi^{r_{\an}} && \pmod{\pi^{r_\an + 1}} \\
 & \equiv  (-1)^{r_{\an} + 1}\sL_{\an}^*(\chi)(\pi^{r_\an} - Y^{r_\an}) \tilde{\sF} && \pmod{\pi^{r_\an + 1}},
 \end{alignat*}
hence
\begin{equation} \label{e:upy}
 \prod_{i=1}^{r} (U_{\fp} - 1) \sF \equiv  (-1)^{r_{\an} + 1}\sL_{\an}^*(\chi)(\pi^{r_\an} - Y^{r_\an}) \sF  \pmod{\pi^{r_\an + 1}}.
 \end{equation}

  Combining (\ref{e:tly})--(\ref{e:upy}), we see that there is a surjective $\Lambda_{(1)}$-algebra homomorphism \begin{equation}  \label{e:whom}
   W_2 \longrightarrow \varphi(\T) \otimes_{\cO_E} E \end{equation} such that
  $y$ maps to the image of $Y$ in $\End_{E[\pi]/\pi^{r_\an + 1}} \cH$ and $\epsilon_i$ maps to the image of $ U_{\fp_i}- 1$.  For future reference, we note that we have not yet used the condition $\nu_1(\cW) \neq 0$ in this proof.

  To conclude the proof,
  we must  demonstrate that the homomorphism (\ref{e:whom}) is an injection, which we again accomplish by counting 
  dimensions.  The algebra $W_2$ has dimension $2^{r_{\an} } + 2 r_{\an} - 2$ as an $E$-vector space and is generated by
   the images of \[ 1, \pi, \pi^2, \dots, \pi^{r_\an - 1}, y, y^2, \cdots, y^{r_\an}, \] and the products
  $\epsilon_J = \prod_{j \in J} \epsilon_i$ for all subsets $J \subset R, J \neq \phi, R$.

  First suppose $\nu_1(\cW) \neq -1$  (in addition to the assumption $\nu_1(\cW) \neq 0$ of the theorem) and suppose we have an $E$-linear combination of the
  forms \[ \{ \pi^i \sF\}_{i=0}^{r_\an - 1} \cup \{ Y^i \sF\}_{i=1}^{r_\an} \cup \left\{ \prod_{j \in J} (U_{\fp_j} - 1) \sF \right\}_{J \neq \phi, R} \subset \cH \]
  that vanishes. We must show that each of the coefficients in this linear combination is zero.
Now $\sF$ does not vanish at $k=1$, i.e. $\nu_1(\sF) = (1 + \nu_1(\cW))E_1(1, \chi_S) \neq 0$, and it is the only form in our list with this property; therefore its coefficient in our linear combination must be zero.  Next we consider the two order 1 terms in our list, namely $\pi \cF$ and $Y \cF$.  Suppose the coefficients of these two terms in our linear combination are $\alpha$ and $\beta$.  Then by considering leading terms, we must have $\alpha(1 + \nu_1(\cW)) + \beta \nu_1(\cW) = 0$.  However by applying $Y$ and then considering leading terms, we also find $\alpha + \beta = 0$.  These two equations imply that $\alpha = \beta = 0$.  Continuing in this fashion, we see that all the coefficients of the terms in our linear combination with order less than $r_{\an}$ must vanish.  It remains to prove that the image of the  forms $Y^{r_\an} \sF$ and
$ \left\{ \prod_{j \in J} (U_{\fp_j} - 1) \sF \right\}_{J \neq \phi, R} $ in $\cH$ are linearly independent over $E$.  However, modulo $\pi^{r_\an+1}$, these forms are congruent up to non-zero scalars to the forms $\pi^{r_\an} E_1(1, \chi_J)$ for $J \subset R$, $J \neq \phi$.  As noted earlier, these forms are linearly independent.

If $\nu_1(\cW) = -1$, a similar  argument goes through.  The minimal order forms in our list are $\sF$ and $Y \sF$; these each have order 1 and their leading terms (i.e.\ their images in $\cH$ modulo $\pi^2$) are linearly independent.  This implies that their coefficients in our linear combination are zero.  The next minimal order forms are $\pi \sF$ and $Y^2 \sF$, which each have order $2$ and have leading terms that are linearly independent.  Continuing in this way, we are reduced to proving that the order $r_\an$ forms $\pi^{r_\an - 1} \sF, Y^{r_\an} \sF,$ and $ \left\{ \prod_{j \in J} (U_{\fp_j} - 1) \sF \right\}_{J \neq \phi, R} $ are linearly independent modulo $\pi^{r_\an+1}$.  The linear independence of all but the first of these forms follows exactly as in the previous case.  We must therefore prove that  $\pi^{r_\an - 1} \sF$ cannot be written as a linear combination of $Y^{r_\an} \sF$ and $ \left\{ \prod_{j \in J} (U_{\fp_j} - 1) \sF \right\}_{J \neq \phi, R} $ modulo $\pi^{r_\an + 1}$.  However, applying $Y$ to such a putative linear combination, we would find that
 $\pi^{r_\an - 1} Y \sF \equiv 0 \pmod{\pi^{r_\an + 1}}$ since $Y$ annihilates all of the forms $Y^{r_\an} \sF$ and $ \left\{ \prod_{j \in J} (U_{\fp_j} - 1) \sF \right\}_{J \neq \phi, R} $ modulo $\pi^{r_\an + 1}$.  But \[ \pi^{r_\an - 1} Y \sF \equiv \pi^{r_\an} E_1(1, \chi_S) \nu_1(\cW) \not\equiv 0 \pmod{\pi^{r_\an+1}}. \]
This concludes the proof.
\end{proof}

\begin{remark} Note that when $r = 1$ and $w = \nu_1(W) \neq 0, -1$, there is a natural $\Lambda_{(1)}$-algebra homomorphism $W_2 \longrightarrow E[\pi]/\pi^2$ given by
\[ y \mapsto \pi \cdot w / (w+1), \qquad  \epsilon \mapsto \sL_{\an}(\chi) \pi /(w+1). \]
We therefore obtain a $\Lambda$-algebra homomorphism $\T \rightarrow E[\pi]/\pi^2$ such that:
\begin{alignat*}{2}
T_\fl &\mapsto   1 + \chi(\fl) + \frac{\chi(\fl) + w}{1 + w} (\log \langle \N\fl \rangle) T,  \qquad &&
\fl \nmid \fn p  \\
U_\fl &\mapsto 1, && \fl \mid \fn  \\
U_{\fp} &\mapsto 1 + \frac{\sL_{\an}}{1 + w}T, && R = S_p = \{\fp\}.
\end{alignat*}
This is the exactly homomorphism constructed in case 2 in \cite{ddp}.
\end{remark}

\subsection{Hecke Action in Case 3: $R' = \phi, \nu_1(\cw) = 0$ }

Suppose that $\cW$ has a zero at $k=1$, i.e. $r_{\an}(\chi) > r_{\an}(\chi^{-1})$.
For notational simplicity we write $s = r_{\an}(\chi)$ and $t = r_{\an}(\chi^{-1}).$
Define the $\Lambda_{(1)}$-algebra \[ W_3 = E[\pi, \epsilon_1, \dotsc, \epsilon_r, y]/I_{W_3} \]
  where
   \begin{align*}
    I_{W_3} = & (\pi^{s+1}, y^{t + 1}, y(\pi - y), \pi^{t}\cW - y^{t},  \\
    & \ \ \ \  \epsilon_i^2, \epsilon_i \pi, \epsilon_i y , \epsilon_1\epsilon_2 \cdots \epsilon_r + (-1)^{s} \sL_{\an}^*(\chi)\pi^{s}).
    \end{align*}
\begin{theorem}   \label{t:case3}
  Suppose $R'$ is empty and that $\cW$ has a zero of order $s - t \ge 1$.
  There exists a $\Lambda$-algebra homomorphism
  \[ \varphi \colon \T \longrightarrow W_3 \]
  such that
  \begin{alignat*}{2}
  T_\fl & \mapsto 1 + \chi \epsilon(\fl) + (\chi(\fl) - 1)\frac{1 - \epsilon(\fl)}{\pi} y\quad  &&\text{ for } \fl \nmid \fn p \\
  U_\fl & \mapsto 1  + \frac{\epsilon(\fl) - 1}{\pi} y \quad && \text{ for } \fl \mid \fn, \text{ and } \\
  U_{\fp_i} & \mapsto 1 + \epsilon_i.
  \end{alignat*}
    \end{theorem}

\begin{proof}
As noted earlier, the proof of Theorem~\ref{t:case2} carries through without the use of the assumption
$\nu_1(\cW) \neq 0$ up through the construction of the homomorphism~(\ref{e:whom}).  It is the injectivity of this homomorphism that used the condition $\nu_1(\cW) \neq 0$.  Indeed, if $\nu_1(\cW) = 0$ as we are currently assuming, then (\ref{e:whom}) is not injective.  We have
\[ Y^{t+1} \tilde{\sF} \equiv \pi^{t+1} \sE(\chi, 1) \cW \equiv 0 \pmod{\pi^{s+ 1}} \]
since $\pi^{s-t} \mid \cW$, hence $Y^{t+1} \sF \equiv 0 \pmod{\pi^{s+1}}.$  Furthermore
\[ Y^{t} \tilde{\sF} \equiv \pi^{t} \sE(\chi, 1) \cW \equiv \pi^t \cW \tilde{\sF} \pmod{\pi^{s+ 1}}. \]
It follows that the homomorphism (\ref{e:whom}) factors through the quotient $W_3$ of $W_2$, and to conclude the proof it remains to show that the induced map $W_3 \longrightarrow \varphi(\T) \otimes_{\cO_E} E$ is injective.
For this it suffices to show that the forms
\[ \{\pi^i \sF\}_{i=0}^{s} \cup \{ Y^i \sF\}_{i=1}^{t-1} \cup \left\{ \prod_{j \in J} (U_{\fp_j} - 1) \sF\right\}_{J \subset R, J \neq \phi, R} \]
are $E$-linearly independent modulo $\pi^{s+1}$.  The demonstration of this fact is similar to the previous cases and left to the reader.
\end{proof}

\section{Construction of a Cohomology Class} \label{s:ccc}

We write  \[ \varphi\colon \T \longrightarrow W \] where $W = W_1, W_2,$ or $W_3$ in cases 1, 2, and 3, respectively,
for the homomorphism $\varphi$  given in Theorems~\ref{t:case1}, \ref{t:case2}, and \ref{t:case3}.
We write $\fm_W$ for the maximal ideal of $W$ and $\fm \subset \T$ for the kernel of the composition
 \[ \varphi\colon \T \longrightarrow W \longrightarrow W/\fm_{W} \cong E. \]
The height 1 prime ideal $\fm$ is generated by $T \in \Lambda$, $T_\fl - (1 + \chi(\fl))$ for $\fl \nmid \fn p$ and $U_\fl - 1$ for $\fl \mid \fn p$.

Let $\T_{(\fm)}$ denote the localization of $\T$ at the prime ideal $\fm$.  Let $L = \Frac(\T_{(\fm)})$ denote the total ring of fractions
of the local ring $\T_{(\fm)}$.  Since the tame character $\chi$ in our space of Hida families has conductor equal
to the tame level $\fn$ of our families, there are no $\fn$-old forms and therefore $\T_{(\fm)}$ is reduced.
 This simple yet crucial observation was not mentioned in \cite{ddp}; we thank H.~Hida for pointing it out to us and refer the reader
 to \cite[Proof of Theorem 3.6 and Corollary 3.7, pp.\ 381--382]{hida} for further details.
As a result, we have a canonical injection $\T_{(\fm)} \rightarrow L$  where $L$ is isomorphic to a product of fields
\begin{equation} \label{e:ldecomp} L = \prod_{i=1}^t L_{\sH_i}. \end{equation}
Each $L_{\sH_i}$ is a finite extension of $\Frac(\Lambda)$ and corresponds to a cuspidal Hida eigenfamily $\sH_i$.
For an integral ideal $\fa \subset \cO_F$, the normalized Fourier coefficient $c(\fa, \sH_i)$ is equal to the image in $L_{\sH_i}$ of the Hecke operator $T_\fa$.
These coefficients generate a finite local $\Lambda$-subalgebra of $L_{\sH_i}$ that we denote $\Lambda_{\sH_i}$ and call the Hecke algebra of $\sH_i$.
The image of $\T_{(\fm)}$ in $L_{\sH_i}$ is the localization of $\Lambda_{\sH_i}$ at a height 1 prime ideal $\fm_{\sH_i}$ lying above $(T) \subset \Lambda$, and
the explicit description of the homomorphism $\varphi$ implies that for prime ideals $\fl \subset \cO_F$ we have
 \begin{equation}
 \begin{alignedat}{2}
 c(\fl, \sH_i) & \equiv  1 + \chi(\fl) && \pmod{\fm_{\sH_i}}  \text{ for } \fl \nmid \fn p, \label{e:clambda} \\
  c(\fl, \sH_i) &\equiv 1  && \pmod{\fm_{\sH_i}} \text{ for } \fl \mid \fn p. 
  \end{alignedat}
  \end{equation}
These congruences simply state that the specialization of $\sH_i$ at  the prime ideal $\fm_{\sH_i}$  is the weight 1 form $E_1(1, \chi_S)$.

\subsection{Representations Associated to Hida Families}

As above, let  $\sH$ denote a cuspidal Hida eigenfamily specializing at a weight 1 prime ideal $\fm_\sH \subset \Lambda_\sH$ to the form $E_1(1, \chi_S)$ (i.e.\ satisfying (\ref{e:clambda})).
   Let $L_\sH = \Frac(\Lambda_\sH)$ denote the fraction field of $\Lambda_\sH$.
The following theorem  (\cite[Theorems 2 and 4]{wiles2}) of Hida and Wiles is crucial for the construction of our cohomology class.

\begin{theorem}[Hida, Wiles] There exists a continuous irreducible Galois representation
\[ \rho_\sH \colon G_F \longrightarrow \GL_2(L_\sH) \]
where $L_\sH$ is endowed with the $\Lambda$-adic topology (i.e.\ the topology induced by the maximal ideal $(\pi_E, T)$ of  $\Lambda$, where $\pi_E$ is a uniformizer for $E$),
such that:
\begin{enumerate}
\item  $\rho_\sH$ is unramified outside $\fn p$;
\item for  primes $\fl \nmid \fn p$, the characteristic polynomial
of $\rho_{\sH}(\Frob_\fl)$ is
\begin{equation} \label{e:charrho}
\chr(\rho_{\sH}(\Frob_\fl))(x) = x^2 - c(\fl, \sH) x + \chi \epsilon (\fl);
\end{equation}
\item for all $\fp \mid p$, we have \begin{equation}
\label{e:localbasis}
 \rho_{\sH}|_{G_\fp} \sim \mat{\chi \epsilon \eta_{\fp, \sH}^{-1}}{*}{0}{\eta_{\fp, \sH}},
 \end{equation}
 where $\eta_{\fp, \sH} \colon G_\fp \longrightarrow \Lambda_{\sH}^*$ is unramified and $\eta_{\fp, \sH}(\rec(\varpi^{-1})) = c(\fp, \sH)$.  Here $\varpi \in F_\fp^*$ is a uniformizer and $\rec \colon F_\fp^* \longrightarrow G_\fp^{\ab}$ is the local Artin reciprocity map.
\end{enumerate}
\end{theorem}

  Note that by (\ref{e:charrho}) we have $\chr(\rho_{\sH}(\Frob_\lambda))(x) \in \Lambda_{\sH}[x]$, and hence by Cebotarev
 we have \begin{equation} \label{e:chri}
  \chr(\rho_{\sH}(\sigma))(x) \in \Lambda_{\sH}[x] \qquad  \end{equation}
  for all  $\sigma \in G_F$. Moreover by
 (\ref{e:clambda}), (\ref{e:charrho}), and another application of Cebotarev we have
\begin{equation} \label{e:factorrho}
 \chr(\rho_{\sH}(\sigma))(x) \equiv (x - 1)(x - \chi(\sigma))  \pmod{\fm_\sH}
 \end{equation}
for all $\sigma \in G_F$. Note that in applying Cebotarev and the continuity of $\rho_\sH$ to deduce (\ref{e:chri}) and (\ref{e:factorrho}), we are using the fact that
$\Lambda_{\sH}$ and $\fm_\sH$ are finitely generated $\Lambda$-modules and hence are closed in the $\Lambda$-adic topology on $L_{\sH}$.

In order to rigidify the representation $\rho_{\sH}$, we choose an element $\tau \in G_F$ such that $\chi(\tau) \neq 1$.
Let $\Lambda_{\fm_\sH}$ denote the completion of the localization of $\Lambda_{\sH}$ at $\fm_\sH$ with respect to its maximal ideal.  We denote the maximal ideal of $\Lambda_{\fm_\sH}$  by $\hat{\fm}_\sH = \fm_\sH \Lambda_{\fm_\sH}$.
By (\ref{e:factorrho}) and Hensel's Lemma, $\rho_{\sH}(\tau)$ has  distinct eigenvalues $\lambda_1, \lambda_2 \in \Lambda_{\fm_\sH}$ such that
$\lambda_1 \equiv 1 \pmod{\hat{\fm}_\sH}$ and $\lambda_2 \equiv \chi(\tau) \pmod{\hat{\fm}_\sH}.$  After extending scalars to $L_{\fm_{\sH}} = \Frac(\Lambda_{\fm_\sH})$, we can choose a basis for our representation consisting of eigenvectors for $\rho_{\sH}(\tau)$, i.e.\ such that
\begin{equation} \label{e:rhohtau}
\rho_{\sH}(\tau) = \mat{\lambda_1}{0}{0}{\lambda_2}. \end{equation}

In the next section, we will demonstrate how to define a cohomology class using the upper right entries of the representation $\rho_\sH$ in this basis as $\sH$ ranges over the $\sH_i$. Ribet showed how to gain local information about this cohomology class by comparing the ``global" basis satisfying (\ref{e:rhohtau}) to the ``local" basis indicated in (\ref{e:localbasis}). This argument, which Mazur \cite{m} has called ``Ribet's Wrench," does not succeed in our context if the global basis and local basis are the same.  We must show, therefore, that $\tau$ can be chosen so that its eigenvectors do not agree with the eigenvectors of $\rho_{\sH}(G_\fp)$ for any $\fp \mid p$.  Furthermore, we must do this simultaneously for all the finitely many $\sH$ that  occur.

\begin{lemma} \label{l:ev}
Let $v \in L_{\fm_\sH}^2$ be a nonzero vector in the representation space of $\rho_{\sH}$, and let $G_v \subset G_F$ denote the subgroup of elements $\sigma$ such that
$v$ is an eigenvector for $\rho_\sH(\sigma)$.  If $\chi(G_v) \neq 1,$ then $G_v$ has infinite index in $G_F$.
\end{lemma}

\begin{proof}
Fix a $\tau \in G_v$ such that $\chi(\tau) \neq 1$.  As above let  $\lambda_1, \lambda_2 \in \Lambda_{\fm_\sH}$ be the eigenvalues of $\rho_\sH(\tau)$
such that
$\lambda_1 \equiv 1 \pmod{\hat{\fm}_\sH}$ and $\lambda_2 \equiv \chi(\tau) \pmod{\hat{\fm}_\sH}.$
Choose a basis for $\rho_\sH(\sigma) = \mat{a_\sH(\sigma)}{b_\sH(\sigma)}{c_\sH(\sigma)}{d_\sH(\sigma)}$ whose first vector is $v$ and such that $\rho_\sH(\tau)$ is diagonal; hence
\begin{equation} \label{e:rhotauH}
 \rho_\sH(\tau) = \mat{\lambda_1}{0}{0}{\lambda_2} \quad \text{ or } \quad \rho_\sH(\tau) = \mat{\lambda_2}{0}{0}{\lambda_1}.
 \end{equation}
 Let us for the moment assume that the first of these cases holds, as the second case is similar and proceeds in the same fashion.

By (\ref{e:chri}) we have
\[ a_\sH(\sigma) + d_{\sH}(\sigma) = \tr \rho_{\sH}(\sigma) \in \Lambda_{\sH} \subset \Lambda_{\fm_\sH} \]
for any $\sigma \in G_F$
and moreover by (\ref{e:factorrho}) we have
 \begin{equation}
  a_{\sH}(\sigma) + d_{\sH}(\sigma) \equiv 1 + \chi(\sigma) \pmod{\hat{\fm}_{\sH}}.  \label{e:adcong}
  \end{equation}
Now by (\ref{e:rhotauH}):
\begin{equation}
\begin{alignedat}{2}
a_{\sH}(\tau) &= \lambda_1 \equiv 1 && \pmod{\hat{\fm}_{\sH}},  \label{e:adhtau} \\
d_{\sH}(\tau) &= \lambda_2 \equiv \chi(\tau) && \pmod{\hat{\fm}_{\sH}}. 
\end{alignedat}
\end{equation}
We have
\begin{alignat}{2}
 1 + \chi(\sigma)\chi(\tau) & \equiv a_{\sH}(\sigma \tau) + d_{\sH}(\sigma\tau)   && \pmod{\hat{\fm}_{\sH}} \label{e:stsum} \\
& \equiv a_{\sH}(\sigma) + d_{\sH}(\sigma) \chi(\tau) && \pmod{\hat{\fm}_{\sH}}, \label{e:stbreak}
\end{alignat}
where (\ref{e:stsum}) follows from (\ref{e:adcong}) with $\sigma$ replaced by $\sigma\tau$ and (\ref{e:stbreak}) follows from (\ref{e:adhtau}).  Now (\ref{e:adcong}) and (\ref{e:stbreak}) imply that
\begin{equation} \label{e:adcong2}
 a_{\sH}(\sigma) \equiv 1 \pmod{\hat{\fm}_{\sH}}, \qquad d_{\sH}(\sigma) \equiv \chi(\sigma) \pmod{\hat{\fm}_{\sH}}.
\end{equation}
(In particular, $a_{\sH}(\sigma), d_{\sH}(\sigma) \in \Lambda_{\fm_\sH}$.)

Let $C_0$ denote the $\Lambda_{\sH}$-module generated by the elements $c_{\sH}(\sigma)$ for $\sigma \in G_F$ and let
$C$ denote the $\Lambda_{\fm_\sH}$-module generated by the  $c_{\sH}(\sigma)$.
The continuity of $\rho_\sH$ and the compactness of $G_F$ imply that $C_0$ is compact.  It follows that $C_0$ is a finitely-generated $\Lambda_\sH$-module, and hence that
$C$ is a finitely generated $\Lambda_{\fm_\sH}$-module.

  The equation
\[ c_{\sH}(\sigma \tau) = c_{\sH}(\sigma)a_{\sH}(\tau)  + d_{\sH}(\sigma)c_{\sH}(\tau) \]
together with (\ref{e:adcong2}) implies that  $\overline{c}_\sH(\sigma) \in C/\hat{\fm}_{\sH} C$ is a 1-cocycle representing a cohomology class $\kappa \in H^1(G_F, C/\hat{\fm}_\sH C(\chi))$.

The restriction of  $\kappa$ to $G_v$ clearly vanishes, since $c(G_v) = 0$.
If $G_v$ has finite index in $G_F$, then the inflation-restriction sequence shows that $\kappa$ itself is a trivial cohomology class, i.e.\ we have $\overline{c}_{\sH}(\sigma) = (\chi(\sigma) - 1)x$ for some $x \in C/\hat{\fm}_{\sH}C$.  Evaluating at $\sigma = \tau$ we see that in fact $x=0$,
i.e.\ the image of $c_\sH$ in  $C/\hat{\fm}_\sH C$ is zero.  However, the $c_\sH(\sigma)$ generate the module  $C/\hat{\fm}_\sH C$ by definition.
 Therefore $C/\hat{\fm}_\sH C = 0$ and hence by Nakayama's Lemma, we must have $C=0$; hence $c_\sH$ is zero as a function on $G_F$.  This contradicts the irreducibility of $\rho_\sH$,
 and hence $G_v$ must have infinite index in $G_F$.

 If the second case in (\ref{e:rhotauH}) holds, then $\overline{c}_\sH(\sigma) \in C/\hat{\fm}_{\sH} C$ represents a cohomology class $\kappa \in H^1(G_F, C/\hat{\fm}_\sH C(\chi^{-1}))$ and the same argument goes through.
\end{proof}

For each prime $\fp \in R$ and each Hida family $\sH$ as above, let $v_{\fp, \sH} \in L_{\fm_\sH}^2$ be the eigenvector for $\rho_\sH(G_{\fp})$.

\begin{lemma} \label{l:tauexists} There exists a $\tau \in G_F$ such that $\chi(\tau) \neq 1$ and such that $v_{\fp, \sH}$ is not an eigenvector
for $\rho_\sH(\tau)$ for all $\sH$ and $\fp$.
\end{lemma}

\begin{proof}
In the notation of Lemma~\ref{l:ev}, we must show that there exists a $\tau \in G_F$ such that $\chi(\tau) \neq 1$
and $\tau \not \in G_{v_{\fp, \sH}}$ for all $\fp$ and $\sH$.  Label the $v_{\fp, \sH}$ such that $\chi(G_{v_{\fp, \sH}}) \neq 1$ as
$v_1, \dotsc, v_n$ and the remaining $v_{\fp, \sH}$ as $v_{n+1}, \dotsc, v_m$.

 We construct $\tau$ inductively.
Let $\tau_0 \in \Gal(H/F)$ be nontrivial, so $\chi(\tau_0) \neq 1$.  Let $H_0 = H$.
We define $\tau_i$ for   $i = 1, \dotsc, n$ recursively as follows.
 Since $G_{v_i}$ has infinite index in $G_F$ by Lemma~\ref{l:ev}, there exists an $\alpha_i \not \in H_{i-1}$ in
the fixed field of $G_{v_i}$ acting on $\overline{F}$.  Let $H_i$ be the Galois closure of $H(\alpha_i)$ over $F$, and let $\tau_i$ be an element of $\Gal(H_i/F)$ such that $\tau_i|_{H_{i-1}} = \tau_{i-1}$ and $\tau_i(\alpha_i) \neq \alpha_i$.  Then any $\tau \in G_F$ restricting to $\tau_i$ will satisfy $\chi(\tau) \neq 1$ and $\tau \not \in G_{v_i}$, since $\tau$ acts nontrivially on the fixed field of $G_{v_i}$.

After defining $\tau_1, \dotsc, \tau_n$ in this way, let $\tau \in G_F$ be any element restricting to $\tau_n$ on $H_n$.
 Then by construction, $\chi(\tau) \neq 1$ and $\tau \not \in G_{v_i}$ for $i=1, \dots, n$.  Clearly $\tau \not \in G_{v_i}$ for $i = n+1, \dots m$, since $\chi(\tau) \neq 1$ and $\chi(G_{v_i}) = 1$ for these $i$.  This concludes the proof.
\end{proof}

\subsection{Construction of the Cohomology Class} \label{s:kappa}

Recall that $\T_{(\fm)}$ denotes the localization of $\T$ at the prime ideal $\fm$, and that \[ L = \prod_{i=1}^t L_{\sH_i} \] denotes its total ring of fractions.  Let $\overline{\T}$ denote the image of $\T$ in $\T_{(\fm)}$.
The product of the Galois representations $\rho_{\sH_i}$ for $i=1, \dotsc, t$
yields a continuous Galois representation
\[ \rho\colon G_F \longrightarrow \GL_2(L), \]
where $L$ is endowed with the $\Lambda$-adic topology,
satisfying:
\begin{enumerate}
\item  $\rho$ is unramified outside $\fn p$;
\item for  primes $\fl \nmid \fn p$, the characteristic polynomial
of $\rho(\Frob_\fl)$ is
\begin{equation} \label{e:charrho2}
\chr(\rho_{\sH}(\Frob_\fl))(x) = x^2 - \overline{T}_\fl x + \chi \epsilon (\fl),
\end{equation}
where $\overline{T}_\fl$ denotes the image of $T_\fl$ in $\overline{\T}$;
\item for all $\fp \mid p$, we have \begin{equation}
\label{e:localbasis2}
 \rho|_{G_\fp} \sim \mat{\chi \eta_\fp^{-1} \epsilon}{*}{0}{\eta_\fp},
 \end{equation}
 where $\eta \colon G_\fp \longrightarrow \overline{\T}^*$ is unramified and $\eta_\fp(\rec(\varpi^{-1})) = \overline{U}_\fp$.
 \end{enumerate}

  Let $\T_\fm$ denote the completion of $\T_{(\fm)}$ with respect to its maximal ideal $\fm \T_{(\fm)}$.   We write $\hat{\fm} = \fm \T_\fm$ for the maximal ideal of $\T_\fm$.
 Let $\tau \in G_F$ satisfy the conditions of Lemma~\ref{l:tauexists}.  By Hensel's Lemma, there exist unique roots $\lambda_1, \lambda_2 \in \T_\fm$ of the characteristic polynomial
  of $\rho(\tau)$ such that $\lambda_1 \equiv 1 \pmod{\fm}, \lambda_2 \equiv \chi(\tau) \pmod{\fm}$.  We extend scalars for the representation $\rho$ to $L_\fm = \Frac(\T_\fm)$ and choose a basis for the representation consisting of the associated eigenvectors for $\rho(\tau)$, i.e. such that
  \begin{equation} \label{e:rhotau}
   \rho(\tau) = \mat{\lambda_1}{0}{0}{\lambda_2}.
   \end{equation}

We can now construct our desired cohomology class following the method introduced in the proof of Lemma \ref{l:ev}.  Write $\rho(\sigma) = \mat{a(\sigma)}{b(\sigma)}{c(\sigma)}{d(\sigma)}.$  Using (\ref{e:charrho2}) and the fact that $\overline{T}_\fl \equiv 1 + \chi(\lambda) \pmod{\fm} $, it follows from Cebotarev that
\begin{equation} \label{e:aplusd1}
 a(\sigma) + d(\sigma) \in \overline{\T} \subset \T_\fm
 \end{equation} and
 \begin{equation} a(\sigma) + d(\sigma) \equiv 1 + \chi(\sigma) \pmod{\fm \overline{\T}}.  \label{e:aplusd2}
 \end{equation}
 Our applications of Cebotarev and the continuity of $\rho$ to deduce (\ref{e:aplusd1}) and (\ref{e:aplusd2}) rely on the fact that $\T$ and $\fm \subset \T$ (and hence their images in $\T_{(\fm)}$) are finitely generated $\Lambda$-modules and  are therefore closed in the $\Lambda$-adic topology.

Following the argument from (\ref{e:adcong})--(\ref{e:adcong2}) and using (\ref{e:rhotau}), we deduce that $a(\sigma), d(\sigma) \in  \T_\fm$ and
\begin{equation} \label{e:adcong3}
 a(\sigma) \equiv 1 \pmod{\hat{\fm}}, \qquad d(\sigma) \equiv \chi(\sigma) \pmod{\hat{\fm}}.
 \end{equation}
Now let $B$ denote the $\T_\fm$-module generated by the $b(\sigma)$ for $\sigma \in G_F$. Repeating the compactness argument from the proof of Lemma~\ref{l:ev} shows that $B$ is a finitely generated $\T_\fm$-module.   Define the $E$-vector space $\overline{B} = B/\hat{\fm} B$ and let $\overline{b}(\sigma)$ denote the image of $b(\sigma)$ in $\overline{B}$.  The equation
\[ b(\sigma \sigma') = a(\sigma) b(\sigma') + b(\sigma) d(\sigma'), \qquad \sigma, \sigma' \in G_F \] together with (\ref{e:adcong3})
implies that the function
\begin{equation} \label{e:kappadef}
 \kappa(\sigma) = \overline{b}(\sigma) \chi^{-1}(\sigma)\end{equation}
is a 1-cocycle representing a cohomology class  $[\kappa] \in H^1(G_F, \overline{B}(\chi^{-1}))$.

\subsection{Interlude on the Homomorphism $\varphi$}

The local Artin ring $W$ is complete with respect to its maximal ideal $\fm_W$, since $\fm_W^{r_{\an}+1} =0$.  As a result, the homomorphism $\varphi\colon \T \longrightarrow W$ extends canonically to a surjective homomorphism
\[ \varphi_\fm \colon \T_\fm \longrightarrow W. \]

The arguments used to deduce the congruences (\ref{e:adcong3}) can be refined to calculate the images of $a(\sigma)$ and $d(\sigma)$ under the homomorphism $\varphi_\fm$.  The key observation that allows this is the following.  While it is clear that $\varphi_\fm \pmod{\fm_W}$ decomposes as the sum of  two characters (namely, $1$ and $\chi$), the same is in fact true for the full homomorphism $\varphi_\fm$.
In cases 2 and 3, define the ``$\Lambda$-adic cyclotomic character in the variable $y$",
\[  \epsilon_y\colon G_F \longrightarrow W^* \]
to be the character $\epsilon$ with the variable $\pi$ replaced by $y$, i.e.\ if $\epsilon(\sigma) = \sum_{i=1}^{\infty} a_i \pi^i,$ then
\begin{align}
\epsilon_{y}(\sigma) &= \sum_{i=0}^{\infty}a_i y^i  \label{e:ycyc} \\
&= 1 + \frac{\epsilon(\sigma) - 1}{\pi} y. \label{e:ycyc2}
\end{align}
Note that (\ref{e:ycyc}) is a finite sum since $y$ is nilpotent, and (\ref{e:ycyc2}) holds from the relation $\pi y = y^2$ in the ring $W$.
Define $\epsilon_{\pi - y}(\sigma)$ similarly, with $y$ replaced by $\pi - y$.
Define
two homomorphisms
\[ \psi_1, \psi_2 \colon G_F \longrightarrow W^*  \]
as follows:
\begin{align*}
\psi_1(\sigma) &= \begin{cases}
1 & \text{ case 1 } \\
\epsilon_y(\sigma)   & \text{ cases 2 and 3,}
\end{cases} \\
\psi_2(\sigma) &= \begin{cases}
\chi\epsilon(\sigma)  & \text{ case 1 } \\
\chi \epsilon_{\pi - y} (\sigma)& \text{ cases 2 and 3}.
\end{cases}
\end{align*}

\begin{lemma} \label{l:adcongs} We have
 \begin{equation}
   \begin{aligned}
  \varphi_{\fm}( a(\sigma)) &= \psi_1(\sigma) \label{e:finalad} \\
    \varphi_{\fm}(d(\sigma)) &= \psi_2(\sigma). 
    \end{aligned}
    \end{equation}
\end{lemma}

\begin{proof}
A direct computation shows that for $\fl \nmid \fn p$, we have
\begin{equation} \label{e:psi1}
 \varphi_{\fm}(T_\fl) = \psi_1(\Frob_\fl) + \psi_2(\Frob_\fl). \end{equation}
Furthermore, it is easy to see that $\epsilon_y \epsilon_{\pi - y} = \epsilon$ using the relation $\pi y = y^2$, and hence
\begin{equation} \psi_1 \psi_2 = \chi \epsilon. \label{e:psi2}
\end{equation}

Now, (\ref{e:psi1}) implies that
\begin{equation} \label{e:adcong4}
\varphi_{\fm}(a(\sigma) + d(\sigma)) = \psi_1(\sigma) + \psi_2(\sigma)
 \end{equation}
for all $\sigma \in G_F$.  The fact that $\psi_1 \equiv 1 \pmod{\fm_W}$ and $\psi_2 \equiv \chi \pmod{\fm_W}$ along with
\[ \varphi_{\fm}(\chr(\rho(\sigma))(x)) = (x - \psi_1(\sigma))(x - \psi_2(\sigma)), \]
which follows from (\ref{e:psi2}) and (\ref{e:adcong4}), implies that
\begin{align}
\varphi_{\fm}(\lambda_1) &= \psi_1(\tau), \\
\varphi_{\fm}(\lambda_2) &= \psi_2(\tau).
\end{align}
Now (\ref{e:adcong4})  applied with $\sigma\tau$ implies that
\begin{equation} \label{e:adcong5}
\varphi_\fm(a(\sigma))\psi_1(\tau) + \varphi_\fm(d(\sigma))\psi_2(\tau) = \psi_1(\sigma\tau) + \psi_2(\sigma\tau).
\end{equation}
Solving (\ref{e:adcong4}) and (\ref{e:adcong5}) yields (\ref{e:finalad}) as desired.
\end{proof}

\begin{remark} Let $I$ be the kernel of $ \varphi' \colon \T_\fm \longrightarrow E[\pi]/(\pi^{r_{\an}+1})$. As in \S\ref{s:kappa}, Lemma \ref{l:adcongs} can be used to construct a cohomology class $[\tilde{\kappa}]$ in $H^1(G_F, (B/IB)(\psi_1\psi_2^{-1}))$. Applying the arguments of \cite{mw} (see also \cite{skinnercmi}) one can deduce a lower bound for the $E$-dimension of $B/IB$ as follows. Let $J$ (the ``Eisenstein ideal") denote the kernel of the structure map $\Lambda_{(1)} \longrightarrow \T_{\fm}/I$. Then there are isomorphisms $\Lambda_{(1)}/J \cong \T_{\fm}/I \cong E[\pi]/(\pi^{r_{\an}+1})$. Hence $J = (\pi^{r_{\an}+1}) \subset \Lambda_{(1)}$. Let $\Fitt_AM$ denote the initial Fitting ideal of a finitely presented $A$-module $M$. Then
\[
\Fitt_{\Lambda_{(1)}} (B/IB) \ (\text{mod } J) = \Fitt_{\Lambda_{(1)}/J} (B/IB) = \Fitt_{\T_{\fm}/I} (B/IB) = \Fitt_{\T_{\fm}} B \ (\text{mod } I) =0.
\]
The last equality holds because $B$ is a faithful $\T_{\fm}$-module. Hence $\Fitt_{\Lambda_{(1)}} (B/IB) \subset J$ and
\[
\dim_EB/IB \geq \dim_E\Lambda_{(1)}/J = r_{\an}+1.
\]
However, it is unclear if $[\tilde{\kappa}]$ can be used to construct $r$ cyclotomic cohomology classes in $H^1_R(G_F, E(\chi^{-1}))$.
\end{remark}

\subsection{Local Behavior of the Cohomology Class} \label{s:lbcc}

We now study in detail the cohomology class $\kappa$ constructed in \S\ref{s:kappa}.

For each place $\fp \mid p$, there is a basis for which the representation $\rho|_{G_\fp}$ takes the shape given in (\ref{e:localbasis2}).
Let $\mat{A_\fp}{B_\fp}{C_\fp}{D_\fp} \in \GL_2(L_\fm)$ denote the change of basis matrix taking this local basis to our fixed global basis satisfying (\ref{e:rhotau}), i.e. such that
\begin{equation} \label{e:changeofbasis}
 \mat{a(\sigma)}{b(\sigma)}{c(\sigma)}{d(\sigma)} \mat{A_\fp}{B_\fp}{C_\fp}{D_\fp} = \mat{A_\fp}{B_\fp}{C_\fp}{D_\fp} \mat{\chi \eta_\fp^{-1} \epsilon(\sigma)}{*}{0}{\eta_\fp(\sigma)}  \end{equation}
for $\sigma \in G_{\fp}.$

\begin{lemma} \label{l:ac}
The elements $A_\fp$ and $C_\fp$ are invertible in $L_\fm$.
\end{lemma}

\begin{proof}
First note that $\T_{\fm} \subset \prod_{i=1}^{t} \Lambda_{\fm_{\sH_i}}$ and hence \[ L_\fm \subset \prod_{i=1}^{t} L_{\fm_{\sH_i}}, \quad \text{ where }
 L_{\fm_{\sH_i}} = \Frac(\Lambda_{\fm_{\sH_i}} ). \]
 We must show that the projections of $A_\fp$ and $C_\fp$ onto each factor $L_{\fm_{\sH_i}}$ are nonzero for $i = 1, \dotsc, t$.
 But if the image of $A_\fp$ or $C_\fp$ is zero in $L_{\fm_{\sH_i}}$, then it is easy to see that the eigenvector for $\rho_{\sH_i}(G_{\fp})$ acting on $L_{\fm_{\sH_i}}^2$
is an eigenvector for $\rho_{\sH_i}(\tau)$.  But we chose $\tau$ in \S\ref{s:kappa} to satisfy the conditions of Lemma~\ref{l:tauexists}, so  this is not the case.  This proves the result.
\end{proof}

Comparing top left entries of the matrix equation (\ref{e:changeofbasis}) and using Lemma~\ref{l:ac}, we find
\begin{equation} \label{e:main}
b(\sigma) = \frac{A_\fp}{C_\fp}\left(\chi\eta_\fp^{-1} \epsilon(\sigma) - a(\sigma) \right), \qquad \sigma \in G_{\fp}.
\end{equation}

\begin{lemma} \label{l:kappaunr}
The cohomology class $[\kappa] \in H^1(G_F, \overline{B}(\chi^{-1}))$ defined in (\ref{e:kappadef}) is unramified outside $R$.
\end{lemma}

\begin{proof}  It is elementary to see that any cohomology class $[\kappa] \in H^1(G_F, \overline{B}(\chi^{-1}))$ is unramified outside $p$.  Indeed, let $v$ be a place of $F$ not lying above $p$ and let $w$ be the place of $H$ lying above $v$ according to the choice of decomposition group $G_v \subset G_F$.
By inflation-restriction, it suffices to prove that the restriction of $[\kappa]$ to $G_w \subset G_H$ is unramified.  However, since $\chi|_{G_H} = 1$, this restriction is an element  \[ \res_{w} [\kappa] \in H^1(G_w, \overline{B}) = \Hom_{\cts}(G_w^{\ab}, \overline{B}). \]
Now, the image of $I_w$ in $G_w^{\ab}$ is a pro-$\ell$ group where $\ell$ is the prime of $\Q$ below $w$ (or trivial, if $w$ is a complex place)
and $\overline{B}$ is a pro-$p$ group, being a finite dimensional $E$-vector space.  Therefore there are no non-zero continuous homomorphisms between these groups and hence $\res_{I_w}([\kappa]) = 0$.

Next we show that $[\kappa]$ is unramified (in fact locally trivial)
 at primes $\fp \in R'$.  By definition of $R'$, there exists $\tilde{\sigma} \in G_{\fp}$ such that $\chi(\tilde{\sigma}) \neq 1$.
Since $\eta_\fp(\tilde{\sigma}) \equiv \epsilon(\tilde{\sigma}) \equiv a(\tilde{\sigma}) \equiv 1 \pmod{\hat{\fm}}$, it follows that $\chi\eta_\fp^{-1} \epsilon(\tilde{\sigma}) - a(\tilde{\sigma}) \in \T_\fm^*$
and hence by (\ref{e:main}) we have $A_\fp/C_\fp \in B$.  Reducing (\ref{e:main}) modulo $\hat{\fm} B$ we see that $\res_\fp \kappa$ is a coboundary:
\[ \kappa(\sigma) = (1 - \chi^{-1}(\sigma))\overline{A_\fp/C_\fp}, \qquad \sigma \in G_\fp .\]
Therefore $\res_\fp [\kappa] = 0$ as desired.
\end{proof}

\begin{lemma} \label{l:bgen}  The $\T_\fm$-module $B$ is generated by $b(\sigma)$ for all  $\sigma \in I_\fp, \fp \in R$.
\end{lemma}

\begin{proof}
Let $B_I$ be the $\T_\fm$-module generated by $b(\sigma)$ for all  $\sigma \in I_\fp, \fp \in R$.  Let $\overline{B}_I  = B/B_I$.  We want to show that $\overline{B}_I =0$.
Let $[\kappa_I]$ denote the image of the cohomology class $[\kappa]$ in
$H^1(G_F, (\overline{B}_I /\hat{\fm}\overline{B}_I)(\chi^{-1})).$  By Lemma~\ref{l:kappaunr}, the class $[\kappa_I]$  is unramified outside $R$.  But by the definition of $\overline{B}_I$,
the image of $\kappa(\sigma)$ in $\overline{B}_I /\hat{\fm}\overline{B}_I$ is trivial for $\sigma \in I_\fp$, $\fp \in R$, and therefore $[\kappa_I]$ is unramified everywhere.  By Proposition~\ref{p:unr}, it follows that $[\kappa_I] = 0$.  Repeating the argument at the end of Lemma~\ref{l:ev} shows that $\overline{B}_I =0$. Indeed, writing $\kappa_I$ as a coboundary and evaluating at $\tau$ shows that $\kappa_I = 0$ as a function. Yet the values of $\kappa_I$ generate
$\overline{B}_I /\hat{\fm}\overline{B}_I$ and hence $\overline{B}_I /\hat{\fm}\overline{B}_I = 0$.  Since $\overline{B}_I$ is a finitely generated $\T_\fm$-module, Nakayama's Lemma implies that $\overline{B}_I = 0$ as desired.
\end{proof}

\begin{lemma} \label{l:bsub}  Let $R = \{ \fp_1, \dotsc, \fp_r\}$.  We have $B \subset \frac{A_{\fp_1}}{C_{\fp_1}} \hat{\fm} + \cdots + \frac{A_{\fp_r}}{C_{\fp_r}} \hat{\fm}$.
\end{lemma}

\begin{proof}  This follows from Lemma~\ref{l:bgen} and  equation (\ref{e:main}), together with the observation that  for $\fp \in R$, we have $\chi(I_\fp) = 1$ and
 \[ \eta_\fp(\sigma) \equiv \epsilon(\sigma) \equiv a(\sigma) \equiv 1 \pmod{\hat{\fm}}, \qquad \sigma \in I_\fp. \]
 \end{proof}

\section{Computation of the Regulator}

We now assemble the constructions of the previous sections and complete the proof of Theorem~\ref{t:main}, which states that
$\sL_{\an}(\chi) = \sR_p(\chi) $.  Let $\I$ denote the kernel of the homomorphism $\varphi_\fm\colon \T_\fm \longrightarrow W$.

\subsection{Proof of $\sL_{\an}(\chi) = \sR_p(\chi)$ in Cases 1, 2, and 3} \label{s:comp1}

 Let $[\kappa] \in H^1_R(G_F, \overline{B}(\chi^{-1}))$ denote the cohomology class constructed in \S\ref{s:kappa}.
Let $u_1, \dotsc, u_r$ denote an $E$-basis of $U_\chi$.  By Proposition~\ref{p:orth2} and Lemma~\ref{l:kappaunr}, we have
\begin{equation} \label{e:ressumkappa}
 \sum_{i=1}^{r} \res_{\fp_i} \kappa(u_j) = 0 \text{ in } \overline{B} \qquad \text{ for } j = 1, \dotsc, r. \end{equation}
For each fixed $j$, we can write $u_j = \sum_k y_{jk} \otimes e_{jk}$ where $y_{jk} \in \cO_H[1/p]^*$ and $e_{jk} \in E$.
For each $i = 1, \dotsc, r$, let
\[ \sigma_{ij} = \sum_k e_{jk} [ {y}^{(i)}_{jk}] \in E[G_\fp] \]
where ${y}^{(i)}_{jk} \in G_{\fp_i}$ is any element whose image in $G_{\fp_i}^{\ab}$ is equal to the image of $y_{jk}$ under the local Artin reciprocity map (\ref{e:rec})
(as usual we use (\ref{e:localembed}) to embed $\cO_H[1/p]^* \subset F_{\fp_i}^*$).  Then noting that $\chi(G_{\fp_i}) = 1$, we have by definition:
 \[ \res_{\fp_i} \kappa(u_j) = \overline{b}(\sigma_{ij}) \text{ in } \overline{B} \qquad \text{ where }  b(\sigma_{ij}) = \sum_k e_{jk} b({y}^{(i)}_{jk}) \in B. \]
Therefore (\ref{e:ressumkappa}) can be written
\begin{equation} \label{e:bsum}  \sum_{i=1}^{r} b(\sigma_{ij}) \in \hat{\fm} B \qquad \text{for each } j = 1, \dotsc, r.
\end{equation}
Now by (\ref{e:main}), we have
\begin{equation} \label{e:bsij}
b(\sigma_{ij}) = \sum_k e_{jk} \cdot \frac{A_i}{C_i}\left(\eta_i^{-1} \epsilon({y}^{(i)}_{jk}) - a({y}^{(i)}_{jk}) \right)
\end{equation}
where we have written for simplicity $A_i, C_i,$ and $\eta_i$ for $A_{\fp_i}, C_{\fp_i},$ and $\eta_{\fp_i}$.  As we have noted, the term in parenthesis on the right lies in $\hat{\fm}$ since $\eta_i, \epsilon, a$  all lie in $\T_\fm$ and are congruent to 1 modulo $\fm$.  Furthermore we have:
\begin{alignat}{2}
\eta_i^{-1}({y}^{(i)}_{jk}) & = U_{\fp_i}^{o_i(y_{jk})} \equiv 1 + o_i(y_{jk})(U_{\fp_i} - 1) && \pmod{\I} \nonumber \\
 \epsilon({y}^{(i)}_{jk}) &\equiv  1 + \ell_i(y_{jk}) \pi && \pmod{\pi^2}  \nonumber \\
 a({y}_{jk}^{(i)}) & \equiv 1 + a'_i(y_{jk}) && \pmod{(\hat{\fm}^2, \I)}, \label{e:apcong}
 \end{alignat}
 where $a'_i(y_{jk}) \in \hat{\fm}$ is any element such that
\[ \varphi_{\fm}(a'_i(y_{jk})) = \begin{cases}
0 & \text{ case 1 } \\
\ell_i(y_{jk}) y & \text{ cases 2 and 3.}
\end{cases}
\]  The congruence (\ref{e:apcong}) follows from Lemma~\ref{l:adcongs}.
Of course $\pi^2 \in \hat{\fm}^2$.  Therefore
\[ \eta_i^{-1} \epsilon({y}^{(i)}_{jk}) - a({y}^{(i)}_{jk}) \equiv \ell_i(y_{jk})\pi + o_i(y_{jk})(U_{\fp_i} - 1) - a'_i(y_{jk}) \pmod{(\hat{\fm}^2 , \I)}. \]
Hence (\ref{e:bsij}) can be written more simply as
\[ b(\sigma_{ij}) = \frac{A_i}{C_i}\left( \ell_i(u_j)\pi + o_i(u_j)(U_{\fp_i} - 1) - a_i'(u_j) + m_{ij}\right) \]
for some $m_{ij} \in (\hat{\fm}^2, \I)$.  Now in view of Lemma~\ref{l:bsub}, which implies that $\hat{\fm} B \subset \sum_{i=1}^r \frac{A_i}{C_i}\hat{\fm}^2$,  (\ref{e:bsum}) can be written
\[ \sum_{i=1}^{r}  \frac{A_i}{C_i}\left( \ell_i(u_j)\pi + o_i(u_j)(U_{\fp_i} - 1)  - a_i'(u_j)+ m_{ij}\right) = 0  \qquad \text{for each } j = 1 \dotsc, r, \]
after altering the $m_{ij}$ by elements of $\hat{\fm}^2$. It follows that
\[ \det\left(\frac{A_i}{C_i}\left( \ell_i(u_j)\pi + o_i(u_j)(U_{\fp_i} - 1) - a_i'(u_j) + m_{ij}\right)\right)_{i, j = 1, \dotsc, r} = 0 \]
since it is the determinant of a matrix whose rows all sum to 0. Cancelling the constants $\frac{A_i}{C_i}$ (which are invertible by Lemma~\ref{l:ac}) from the rows of this matrix, we obtain
\[ \det\left( \ell_i(u_j)\pi + o_i(u_j)(U_{\fp_i} - 1) - a_i'(u_j) + m_{ij}\right) = 0. \]
This determinant now takes place in the ring $\T_\fm$, and in fact all of its entries lie in the maximal ideal $\hat{\fm}$.  We apply the homomorphism $\varphi_{\fm}$ to this equation to obtain an equation in the ring $W$:
 \begin{alignat}{2}
 \det(\left( \ell_i(u_j)\pi + o_i(u_j)\epsilon_i + n_{ij}\right) = 0 \qquad & \text{case 1,} \nonumber \\
  \det(\left( \ell_i(u_j)(\pi - y) + o_i(u_j)\epsilon_i + n_{ij}\right) = 0 \qquad & \text{cases 2 and 3}, \label{e:detinw}
 \end{alignat}
where $n_{ij} \in \fm_W^2$. Since each entry of this matrix lies in $\fm_W$, it is clear that the $n_{ij}$ do not effect the value of the determinant modulo $\fm_W^{r+1}$.  Finally, using the relations in the ring $W$ (in particular that $\epsilon_i \pi = 0$ and $\epsilon_i y = 0$)  it is easy to calculate these determinants. In case 1 we find
\begin{alignat}{2}
 0 &\equiv \det(\ell_i(u_j)\pi + o_i(u_j)\epsilon_i ) && \pmod{\fm_W^{r+1}} \nonumber \\
 & \equiv \det(\ell_i(u_j)) \pi^r + \det(o_i(u_j)) \epsilon_1 \cdots \epsilon_r  && \pmod{\fm_W^{r+1}} \nonumber \\
 &\equiv  \det(\ell_i(u_j)) \pi^r + \det(o_i(u_j))(-1)^{r_{\an}+1} \sL_{\an}^*(\chi) \pi^{r_{\an}} && \pmod{\fm_W^{r+1}}. \label{e:final}
 \end{alignat}
If $r_{\an} = r$, then $\sL_{\an}^*(\chi) = \sL_{\an}(\chi)$ and since $\pi^r \not \in \fm_W^{r+1}$, it follows that
\[ \sL_{\an}(\chi) = (-1)^r\det(\ell_i(u_j))/ \det(o_i(u_j)) = \sR_p(\chi) \] as desired.
If $r_{\an} > r$, then $ \pi^{r_{\an}} \equiv 0 \pmod{\fm_W^{r+1}}$, so (\ref{e:final}) implies that $\det(\ell_i(u_j)) = 0$, hence $\sR_p(\chi) =0$.  Since $ \sL_{\an}(\chi)  = 0$ in this case as well, we again find $\sL_{\an}(\chi) = \sR_p(\chi)$.

Cases 2 and 3 are nearly identical, once one uses the relations in the ring $W$ to observe that $(\pi - y)^r = \pi^r - y^r \not \in \fm_W^{r+1}$.

\subsection{Proof of   $\sL_{\an}(\chi^{-1}) = \sR_p(\chi^{-1})$ in Case 3}

As noted in \S\ref{s:ccf}, to complete the proof we must show that  $\sL_{\an}(\chi^{-1}) = \sR_p(\chi^{-1})$ in case 3.
For this, we  repeat the arguments from \S\ref{s:lbcc} onward using the ``$c$-cocycle" coming from our representation rather than the ``$b$-cocycle".  To be precise, we let $C$ denote the $\T_{\fm}$-module generated by the elements $c(\sigma)$ for all $\sigma \in G_F$ and write $\overline{C} = C/\hat{\fm}C$.
Then the equation
\[ c(\sigma \sigma') = c(\sigma) a(\sigma) + d(\sigma)c(\sigma'), \qquad \sigma \sigma' \in G_F \]
together with (\ref{e:adcong3}) implies that the function $\overline{c}\colon G_F \rightarrow \overline{C}$ is a 1-cocycle defining a cohomology class \[ [\overline{c}] \in H^1(G_F, \overline{C}(\chi)). \]

The elementary argument at the beginning of the proof of Lemma~\ref{l:kappaunr} shows that $ [\overline{c}] $ is unramified outside $p$, and hence outside $R$ since $R'$ is empty in case 3.  The analogue of (\ref{e:main}), which
is seen by equating lower left entries in (\ref{e:changeofbasis}), is the following:
\[ c(\sigma) = \frac{C_\fp}{A_\fp}\left(\chi \eta_{\fp}^{-1} \epsilon(\sigma) - d(\sigma) \right), \qquad \sigma \in G_\fp. \]

\begin{lemma} \label{e:dcontain}
For $\fp \in R$ and $\sigma \in I_\fp$, we have that
\[ \varphi_\fm( \epsilon(\sigma) - d(\sigma)) \in y W. \]

\end{lemma}

\begin{proof}  Lemma~\ref{l:adcongs} implies that $\varphi_\fm(d(\sigma)) = \epsilon_{\pi - y}(\sigma).$  Using the relation $\pi y = y^2$, it is easy to see that
$\epsilon(\sigma) - \epsilon_{\pi - y}(\sigma) = \epsilon_y(\sigma) - 1$ in $W$.  The result follows.
\end{proof}

From Lemma~\ref{e:dcontain}, the arguments of Lemmas~\ref{l:bgen} and~\ref{l:bsub} apply without change to show that
\[ C \subset \frac{C_{\fp_1}}{A_{\fp_1}} \fy + \cdots + \frac{C_{\fp_r}}{A_{\fp_r}}\fy\]
where $\fy = \varphi_\fm^{-1}(yW)$ is an ideal of $\T_\fm$.

We can next repeat the argument of \S\ref{s:comp1} without change, where now $u_1, \dotsc, u_r$ denotes an $E$-basis of $U_{\chi^{-1}}$.   Noting that  $\varphi_\fm(d(\sigma)) = \chi \epsilon_{\pi - y}(\sigma)$ by Lemma~\ref{l:adcongs} and hence that
\[ \varphi_\fm( \chi \eta_{\fp_i}^{-1} \epsilon(\sigma) - d(\sigma)) =  \epsilon_y(\sigma) - 1 + o_i(\overline{\sigma}) \epsilon_i, \qquad \sigma \in G_{\fp_i}
\]
(where $\overline{\sigma} \in \hat{F}_{\fp_i}^*$ is such that $\rec(\overline{\sigma})$ is the image of $\sigma$ in $G_{\fp}^{\ab}$),
the analogue of (\ref{e:detinw}) is the equation
\[  \det(\left( \ell_i(u_j)y + o_i(u_j)\epsilon_i + n_{ij}\right) = 0 \]
with $n_{ij} \in \fm_W\fy $. We obtain
\begin{equation} \label{e:ypcong}
 \det(\ell_i(u_j)) y^r + \det(o_i(u_j))(-1)^{s+1} \sL_{\an}^*(\chi) \pi^s \equiv 0 \pmod{\fm_W\fy^r}.
 \end{equation}
Note that  in the ring $W = W_3$, we have
\[ y^t = \cW \pi^t = (-1)^{s-t} \frac{\sL_{\an}^*(\chi)}{\sL_{\an}^*(\chi^{-1})} \pi^{s}, \]
hence (\ref{e:ypcong}) can be written
\begin{equation} \label{e:ypcong2}
 \det(\ell_i(u_j)) y^r + \det(o_i(u_j))(-1)^{t+1} \sL_{\an}^*(\chi^{-1}) y^t \equiv 0 \pmod{\fm_W\fy^r}.
 \end{equation}
This congruence yields an equality in the 1-dimensional $E$-vector space $\fy^r / \fm_W \fy^r$, which is generated by the image of $y^r$.
If $t=r$, then $\sL_{\an}^*(\chi) = \sL_{\an}(\chi)$ and we obtain
\[  \det(\ell_i(u_j))  + \det(o_i(u_j))(-1)^{r+1}\sL_{\an}(\chi^{-1}) = 0, \]
hence $\sL_{\an}(\chi^{-1}) = \sR_p(\chi^{-1})$ as desired.  If $t > r$, then $y^t \in \fm_W \fy^r$ so (\ref{e:ypcong2}) yields $ \det(\ell_i(u_j)) = 0$ and hence $\sR_p(\chi^{-1}) = 0$.  Since
$\sL_{\an}(\chi^{-1}) = 0$ in this case as well, we again find $\sL_{\an}(\chi^{-1}) = \sR_p(\chi^{-1})$.  This completes the proof.

\begin{remark}
We  note that this argument fills in a hole at the end of the proof of Theorem 4.4 in \cite{ddp}.  There it was simply suggested without elaboration that switching the roles of $b$ and $c$ yields a cohomology class giving the desired result for $\chi^{-1}$.  This is indeed the case if $r_\an(\chi) = r= 1$, but in the case
$r_\an(\chi) > r_\an(\chi^{-1}) $ one needs a version of the argument presented here and in particular the whole homomorphism $\varphi_\fm$; the homomorphism $\phi_{1+\epsilon}$ constructed in \cite{ddp} does not suffice in case 3.
\end{remark}


\begin{thebibliography}{XXX}

\bibitem{b} {Barsky, D.}
\newblock {\em Fonctions z\^eta $p$-adiques d'une classe de rayon des corps de nombres totalement réels}.
\newblock {Groupe de travail d'analyse ultrametrique}, {\bf 5}(16):1--23, (1977-78).

\bibitem{bkl} {Beilinson, A.~ and Kings, G.~ and Levin, A.}
\newblock{\em Topological polylogarithms and $p$-adic interpolation of $L$-values of totally real fields},
\newblock{preprint: \\} \verb|http://arxiv.org/abs/1410.4741|

\bibitem{burns}{Burns, D.}
\newblock{\em On Derivatives of $p$-adic $L$-series at $s=0$},
\newblock{preprint: \\} \verb|http://www.mth.kcl.ac.uk/staff/dj_burns/dpals-0212.pdf|

\bibitem{cn} {Cassou-Nogu\`es, P.}
\newblock {\em Valeurs aux entiers négatifs des fonctions z\^eta et fonctions z\^eta $p$-adiques.}
\newblock {Invent. Math.}, {\bf 51}(1):29--59, (1979).


\bibitem{pcsd} {Charollois, P.~ and Dasgupta, S.}
\newblock {\em Integral Eisenstein cocycles on $\GL_n$, I:
Sczech's cocycle and $p$-adic $L$-functions of
Totally Real Fields},
\newblock {Cambridge Journal of Mathematics}, {\bf 2}(1):49--90, (2014).

\bibitem{colmez} {Colmez, P.}
\newblock {\em La m\'ethode de Shintani et ses variantes}, unpublished preprint.

\bibitem{das} {Dasgupta, S.}
\newblock {\em Shintani zeta functions and Gross--Stark units for totally real fields},
\newblock {Duke Mathematical Journal} {\bf 143}(2):225--279, (2008).

\bibitem{ddp} {Dasgupta, S.~ and Darmon, H.~ and Pollack, R.}
\newblock {\em Hilbert modular forms and the Gross--Stark conjecture},
\newblock {Ann. of Math.}, {\bf 174}(1):439--484, (2011).

\bibitem{ds}{Dasgupta, S.~and Spiess, M.}
\newblock{\em On the Characteristic Polynomial of Gross's Regulator Matrix},
\newblock{preprint.}

\bibitem{dr} {Deligne, P.~ and Ribet, K.}
\newblock {\em Values of abelian $L$-functions at negative integers over totally real fields},
\newblock {Inventiones Math.}, \textbf{59}:227--286, (1980).

\bibitem{ekw}{Emsalem,  M.~ and Kisilevsky, H.H.~ and Wales, D.}
\newblock {\em Ind\'{e}pendance lin\'{e}aire sur $\overline{\Q}$ de logarithmes $p$-adiques de nombres alg\'{e}briques et rang $p$-adique du groupe des unit\'{e}s d'un corps nombres},
\newblock {J. Number Theory}, {\bf 19}(3):384-391, (1984).

\bibitem{federergross} {Federer, L. J.~ and Gross, B.}
\newblock{\em Regulators and Iwasawa modules},
\newblock{Invent. Math.} {\bf 62}(3):443-457, (1981).

\bibitem{greenberg1} {Greenberg, R.}
\newblock{\em On a certain $l$-adic representation},
\newblock{Invent. Math.}, {\bf 21}: 117-124, (1973).

\bibitem{greenberg} {Greenberg, R.}
\newblock{\em Trivial zeros of $p$-adic $L$-functions},
\newblock{$p$-adic monodromy and the Birch Swinnerton-Dyer conjecture},
\newblock {Contemp. Math. {\bf 165}, Amer. Math. Soc.}, 149--174, 1994.

\bibitem{gross} {Gross, B.}
\newblock {\em $p$-adic $L$-series at $s=0$},
\newblock {Fac. Sci. Univ. Tokyo Sect. IA Math.}, {\bf28}(3):979--994 (1981-82).


\bibitem{hida} {Hida, H.}
\newblock {\em On $p$-adic Hecke Algebras for $\GL_2$ over totally real fields},
\newblock {Ann. of Math.}, {\bf 128}(2):295--384, (1988).


\bibitem{m} {Mazur, B.}
\newblock{\em How can we construct abelian Galois extensions of basic number fields?}
\newblock {Bull. Amer. Math. Soc. (N.S.)}, {\bf 48}(2):155--209, (2011).


\bibitem{mw} {Mazur, B.~ and Wiles, A.}
\newblock{\em Class fields of abelian extensions of $\Q$},
\newblock{Invent. Math.} {\bf 76}(2):179--330, (1984).

\bibitem{ribet} {Ribet, K.}
\newblock {\em A modular construction of unramified $p$-extensions of $\Q(\mu_p)$},
\newblock {Invent. Math.}, {\bf 34}(3):151--162, (1976).

\bibitem{ser} {Serre, J-P.}
\newblock {\em Corps Locaux},
\newblock {Publications de l'Institute de Math\'{e}matique de l'Universit\'{e} de Nancago, VIII}
\newblock {Actualit\'{e}s Sci. Indust., No. 1296. Hermann, Paris 1962. 243 pp.}

\bibitem{skinnercmi} {Skinner, C.}
\newblock{\em Galois representations, Iwasawa theory and special values of $L$-functions},
\newblock {CMI Lecture notes} (2009).

\bibitem{spiesshilb} {Spiess, M.}
\newblock {\em On special zeros of $p$-adic $L$-functions of Hilbert modular forms},
\newblock {Invent. Math.}, {\bf 196}(1):69--138, (2014).

\bibitem{spiess} {Spiess, M.}
\newblock {\em Shintani Cocycles and the vanishing order of $p$-adic Hecke $L$-series at $s=0$},
\newblock {Math. Ann.} {\bf 359}(1-2):239--265, (2014).

\bibitem{stark1}{Stark, H.}
\newblock{\em Values of $L$-functions at $s=1$. I. $L$-functions for quadratic forms},
\newblock{Advances in Math.} {\bf 7}: 301--343 (1971). 

\bibitem{stark2}{Stark, H.}
\newblock{\em Values of $L$-functions at $s=1$. II. Artin $L$-functions with rational characters},
\newblock{Advances in Math.} {\bf 17}, no.\ 1: 60--92 (1975). 

\bibitem{stark3}{Stark, H.}
\newblock{\em Values of $L$-functions at $s=1$. III. Totally real fields and Hilbert's twelfth problem},
\newblock{Advances in Math.} {\bf 22}, no.\ 1: 64--84 (1976).

\bibitem{stark4}{Stark, H.}
\newblock{\em  $L$-functions at $s=1$. IV. First derivatives at $s=0$},
\newblock{Advances in Math.} {\bf 35}, no.\ 3: 197--235 (1980).
 

\bibitem{tatebook} {Tate, J.}
\newblock{\em Les conjectures de Stark sur les fonctions $L$ d'Artin en $s=0$},
\newblock{Progress in Mathematics}, Birkh\"{a}user Boston Inc., Boston, MA, (1984) vol. 47
\newblock{Lecture notes edited by Dominique Bernardi and Norbert Schappacher}.

\bibitem{wiles2} {Wiles, A.}
\newblock{\em On ordinary $\lambda$-adic representations associated to modular forms}
\newblock{Invent. Math.} {\bf 94}:529--573, (1988).

\bibitem{wiles} {Wiles, A.}
\newblock{\em The Iwasawa conjecture for totally real fields},
\newblock{Ann. of Math.} {\bf 131}(3):493--540, (1990).

\bibitem{v} {Ventullo, K.}
\newblock {\em On the rank one abelian Gross--Stark conjecture},
\newblock {Comment. Math. Helv.} {\bf 90}(4):939--963, (2015).

\end{thebibliography}
\end{document}